\documentclass[aop]{imsart}

\RequirePackage{amsthm,amsmath,amsfonts,amssymb}
\RequirePackage[numbers]{natbib}
\RequirePackage[colorlinks,citecolor=blue,urlcolor=blue]{hyperref}
\RequirePackage{graphicx}
\usepackage{scalerel}
\usepackage{float}

\startlocaldefs
\numberwithin{equation}{section}
\theoremstyle{plain}

\newtheorem{theorem}{Theorem}[section]
\newtheorem{lemma}[theorem]{Lemma}
\newtheorem{proposition}[theorem]{Proposition}
\theoremstyle{remark}

\newtheorem{remark}{Remark}

\DeclareMathOperator{\ord}{ord}

\DeclareMathOperator{\Cyl}{Cyl}
\def\smallbox{\mathbin{\scalerel*{\Box}{\cap}}}
\def\bigbox{\mathord{\scalerel*{\Box}{\big(}}}
\DeclareMathOperator*{\bsq}{\bigbox}
\DeclareMathOperator*{\sq}{\smallbox}

\newcommand{\bZ}{{\bf Z}}

\newcommand{\bX}{{\bf X}}

\newcommand{\sfrac}[2]{{\textstyle\frac{#1}{#2}}}

\newcommand{\bx}{{\bf x}}

\newcommand{\by}{{\bf y}}

\newcommand{\bc}{{\bf c}}

\newcommand{\Lspec}{l_{0}}

\newcommand{\MC}{mul\-ti\-pli\-ca\-tive co\-a\-le\-scent}

\newcommand{\cvd}{l^2_{\mbox{{\scriptsize $\searrow$}}}}

\newcommand{\cvt}{l^3_{\mbox{{\scriptsize $\searrow$}}}}

\newcommand{\bA}{{\bf A}}
\newcommand{\R}{\mathbb{R}}
\newcommand{\I}{\mathbb{I}}

\newcommand{\N}{\mathbb{N}}
\newcommand{\p}{\mathbb{P}}
\newcommand{\FF}{\mathcal{F}}
\newcommand{\E}{\mathbb{E}}

\newcommand{\G}{\mathrm{G}}
\newcommand{\cdl}{c\`{a}dl\`{a}g }

\newcommand{\cI}{\mathcal{I}}

\endlocaldefs

\begin{document}

\begin{frontmatter}
\title{On Moments of Multiplicative Coalescents}
%\title{A sample article title with some additional note\thanksref{t1}}
\runtitle{On Moments of Multiplicative Coalescents}
%\thankstext{T1}{A sample additional note to the title.}

\begin{aug}
%%%%%%%%%%%%%%%%%%%%%%%%%%%%%%%%%%%%%%%%%%%%%%%
%% Only one address is permitted per author. %%
%% Only division, organization and e-mail is %%
%% included in the address.                  %%
%% Additional information can be included in %%
%% the Acknowledgments section if necessary. %%
%%%%%%%%%%%%%%%%%%%%%%%%%%%%%%%%%%%%%%%%%%%%%%%
\author[A]{\fnms{Vitalii} \snm{Konarovskyi}\ead[label=e1]{konarovskyi@gmail.com}}
\and
\author[B]{\fnms{Vlada} \snm{Limic}\ead[label=e2]{vlada@math.unistra.fr}}
%%%%%%%%%%%%%%%%%%%%%%%%%%%%%%%%%%%%%%%%%%%%%%
%% Addresses                                %%
%%%%%%%%%%%%%%%%%%%%%%%%%%%%%%%%%%%%%%%%%%%%%%
\address[A]{Faculty of Mathematics, Bielefeld University;
  Institute of Mathematics,
  Leipzig University;\\
  Institute of Mathematics of NAS of Ukraine,
\printead{e1}}

\address[B]{IRMA,
  Strasbourg University,
\printead{e2}}
\end{aug}

\begin{abstract}
  We prove existence of all moments of the multiplicative coalescent at all times. We obtain as byproducts a number of related results which could be of general interest. In particular, we show the finiteness of the second moment of the $l^2$ norm for any extremal eternal version of multiplicative coalescent.
Our techniques are in part inspired by percolation, and in part are based on tools from stochastic analysis, notably the semi-martingale and the excursion theory.  
\end{abstract}

\begin{keyword}[class=MSC]
\kwd[Primary ]{60J90}
\kwd{05C80}
\kwd{60J75}
\kwd[; secondary ]{60G51}
\kwd{60K35}
\end{keyword}

\begin{keyword}
\kwd{Multiplicative coalescent}
\kwd{random graph}
\kwd{excursion}
\kwd{L\'evy process}
\kwd{moment estimates}
\end{keyword}

\end{frontmatter}
%%%%%%%%%%%%%%%%%%%%%%%%%%%%%%%%%%%%%%%%%%%%%%
%% Please use \tableofcontents for articles %%
%% with 50 pages and more                   %%
%%%%%%%%%%%%%%%%%%%%%%%%%%%%%%%%%%%%%%%%%%%%%%
%\tableofcontents

\section{Introduction}%
\label{sec:preliminaries}

The initial motivation for this work came from our recent paper~\cite{Konarovskyi:MC:2020}, where {\em restricted multiplicative merging} (RMM)
was introduced as an important tool for studying novel scaling limits of stochastic block models. 
As a step in our analysis, we needed (see Appendix of \cite{Konarovskyi:MC:2020}) to show that the fourth moment of the $l^2$ norm of a given multiplicative coalescent is finite at any given time. The square of the $l^2$ norm is typically denoted by $S_2$ or $S$ in the literature (see also the definitions preceding the statement of Theorem \ref{the_finiteness_of_fourth_moment}). 
So the above mentioned bound from \cite{Konarovskyi:MC:2020} could be written as $\E( S_2(t) )^2 = \E S^2(t)< \infty$, $t \geq 0$.
Prior to our work, there was no study of higher moments of $S$, or moments of other \MC\ norms (see Lemma \ref{lem_finitenes_of_expected_S_n} or the proof of Proposition \ref{pro_finiteness_of_fourth_moment}), in the literature.
While some of the notation and concepts from our previous article will be initially recalled, this paper is self-contained (in particular, it does not require familiarity with~\cite{Konarovskyi:MC:2020}). 

The multiplicative coalescent is a Markov process on the space $\cvd$ of all infinite sequences $\bx=(x_1,x_2,\dots)$ with $x_1\geq x_2\geq \dots\geq 0$ and $\sum_{ j } x_j^2<\infty$ equipped with $l^2$-norm $\|\cdot \|$. It describes the evolution of masses of countably many blocks evolving according to the following dynamics:
\begin{align*}
  & \mbox{each pair of blocks of mass $x_i$ and $x_j$ merges at rate $x_ix_j$}\\
  & \mbox{into a single block of mass $x_i+x_j$.}
\end{align*}
The process was introduced by David Aldous in~\cite{Aldous:1997} and is a Feller process in $\cvd$ (see Proposition~5 ibid.). The main focus in~\cite{Aldous:1997} was on the construction of a particular eternal version (which is parametrized by $\R $) as the scaling limit of the (component sizes of the) near-critical classical random graphs. Its marginal distribution was characterized as the law of the ordered vector of excursion lengths above past minima of a Brownian motion with parabolic drift.  A year later in~\cite{Aldous:1998}, Aldous and the second author showed that other 
essentially different eternal versions of \MC\ exist,  and gave their full characterization, also via excursion theory.  Extremal eternal \MC s frequently appear as universal scaling limits in various random graph models~\cite{Addario_Berry:2012,Aldous:1998,Aldous:2000,Bhamidi:2014,Bhamidi:2015,Bhamidi:2010,Bhamidi:2012,Joseph:2014,Nachmias:2010,Riordan:2012,Turova:2013}.
However a number of fundamental properties of the \MC s are still not well understood. 

In the sequel we mostly rely on the notation from~\cite{Aldous:1997,Aldous:1998}. 
%In particular, $\cvd$  is a subset of $l^2$ composed of infinite vectors with non-negative components in non-increasing order, and if $\bx$ is a vector in $l^2$ or $\cvd$, then $\|\bx\|$ is its $l^2$-norm.
We reserve the notation $\bX:=(\bX(t), t\geq 0)$ for any \MC\ process, where its initial state will be clear from the context.
Recall that $\bX(t)=(X_1(t),X_2(t),\ldots)$, where $X_j(t)$ is the size of the $j$th largest component at time $t$.
We also denote by $(\bX^*(t), t\in \mathbb{R})$ the standard Aldous' \MC. This and other ``eternal coalescents'' are in fact entrance laws, rather than Markov processes, as they satisfy 
 $\lim_{t\to -\infty} \sum_j (X_j(t))^2 \to 0$.

For any $t$ such that $\bX(t)$ is defined, and any integer $k$ let
$$
S_k(t) := \sum_i (X_i(t))^k.
$$
The natural state space for \MC s is $\cvd$ (all the non-constant eternal \MC s take values in $\cvd\setminus l^1$).
If $\bx \in l^2$ let $\ord(\bx)\in \cvd$ be the infinite vector obtained by listing all the components of $\bx$ in non-increasing order. 
In the sequel we will frequently denote by $\|\bx\|$ the  $l^2$ norm of $\bx$.
Note that then clearly  $\|\bx\|$ = $\|\ord(\bx)\|$, and also that $S (t)\equiv S_2 (t)= \| \bX(t)\|^2$.

The first main result of this paper is the finiteness of all moments of $S$.
\begin{theorem} %finiteness of fourth moment
  \label{the_finiteness_of_fourth_moment}
  Let $\bX(t)$, $t\geq 0$, be a multiplicative coalescent started from $\bx \in \cvd$. Then for every $n \in \N$ and $t\geq 0$ we have
  \[
    \E \|\bX(t)\|^n = \E (S_2(t))^{n/2} < +\infty.
  \]
\end{theorem}
The proof of this general statement has an interesting recursive structure. The estimates on $\E S_n(t)$ and 
$\E S_n(t)S_m(t)$
obtained along the way (see Section~\ref{sub:argument_for_small_times}) are of independent interest.

We next recall in more detail the excursion characterization of the \MC\ entrance laws. For this we introduce the set of parameters
\[
  \cI=\left((0,\infty)\times \R \times \cvt\right)\cup \left(\{ 0 \}\times \R\times \Lspec\right),
\]
where $\Lspec=\cvt \setminus \cvd$, and the processes 
\begin{align*}
  \tilde{W}^{\kappa,t}(s)&= \sqrt{ \kappa }W(s)+t s- \frac{1}{ 2 }\kappa s^2, \quad s\geq 0,\\
  V^{\bc}(s)&= \sum_{ i=1 }^{ \infty } \left( c_i \I_{\left\{ \xi_i\leq s \right\}} -c_i^2s\right), \quad s\geq 0,\\
  W^{\kappa,t,\bc}(s)&= \tilde{W}^{\kappa,t}(s)+V^{\bc}(s), \quad s\geq 0,\\
  B^{\kappa,t,\bc}(s)&= W^{\kappa,t,\bc}(s)-\min\limits_{ r \in [0,s] }W^{\kappa,t,\bc}(r), \quad s\geq 0,
\end{align*}
where $W$ denotes a standard Brownian motion and $(\xi_i)_{i\geq 1}$ is a family of independent exponentially distributed random variables, where $\xi_i$ has rate $c_i$, for each $i\geq 1$.
Note that the process $V^{\bc}$ is well-defined due to $\|\bc\|_3^3:=\sum_{ i=1 }^{ \infty } c_i^3<\infty$. 
 It is well-known that to any extreme eternal (non-constant) \MC\ corresponds a unique $(\kappa,\tau,\bc) \in \cI$ such that this entrance law evaluated at time $t$ is the same as
the decreasingly ordered vector of excursion lengths of $B^{\kappa,t-\tau,\bc}$ (see, e.g.~\cite[Theorem~3]{Aldous:1998}).
Our second goal is to prove the finiteness of the second moment of $l^2$ norm for all the extreme eternal \MC.
The same statement for the standard version was already derived by Aldous in~\cite{Aldous:1997} in two different ways: via excursion theory, and via weak convergence. 

We recall that an excursion $\gamma$ of a non-negative process $B$ is a time interval $[l(\gamma),r(\gamma)]$ such that $B(l(\gamma))=B(r(\gamma))=0$ and $B(s)>0$ for $s \in (l(\gamma),r(\gamma))$. 

\begin{theorem} %finiteness of the second moments of EMC
  \label{the_finiteness_of_the_second_moments_of_emc}
  Let $\Gamma^{\kappa,t,\bc}$ be the set of excursions of $B^{\kappa,t,\bc}$, and let $|\gamma|$ be the length of an excursion $\gamma$. Then for every $(\kappa,t,\bc) \in \cI$, one has
  \[
    \E\sum_{ \gamma \in \Gamma^{\kappa,t,\bc} }   |\gamma|^2<\infty.
  \]
\end{theorem}
	Due to the excursion representation recalled above, it is clear that this statement also says that any extremal eternal \MC\ has a finite second moment at any given time. 
Surprisingly,  our argument for Theorem \ref{the_finiteness_of_the_second_moments_of_emc} with $t<0$ 
is short (and straight-forward), while we had to work much harder to prove the theorem for $t\geq 0$.

\subsection{Graphical construction}%
\label{sub:graph_construction}
We first describe a useful graphical construction and make a link with \cite{Konarovskyi:MC:2020}. 
If $n\in \N$ then $[n]=\{1,2,\ldots,n\}$. 
Here and below the symbol $\bA$ denotes an upper-triangular matrix (or equivalently, a two-parameter family) of 
i.i.d.~exponential (rate $1$) random variables.
While the restricted merging (relation $R$) was typically non-trivial in~\cite{Konarovskyi:MC:2020}, in the present setting we only use the so-called "maximal relation $R^*$".
In other words, there is no restriction on the multiplicative merging, so the family of
evolving random graphs denoted by
$(G_t(\bx;\bA,R^*))_{t,\bx}$ in~\cite{Konarovskyi:MC:2020}
is equal in law to the family of non-uniform random graphs from~\cite{Aldous:1997,Aldous:1998}, also called {\em inhomogenous random graphs}, or {\em rank-1 model} in more recent literature \cite{Bollobas:2007,Bhamidi:2010,Bhamidi:2012}.  
This family of evolving random graphs is a direct continuous-time analogue of Erd\H{o}s-R\'enyi-Stepanov model.
In this general setting there could be (countably) infinitely many particles in the configuration, and the particle masses are arbitrary positive (square-summable) reals. 

We therefore omit $R^*$ from future notation, and frequently we will omit $\bA$ as well.
Let us now fix $\bx \in l^2$ and $t>0$, and describe a somewhat different construction from the one in~\cite{Konarovskyi:MC:2020}.
Set $\N^2_<:=\left\{ (i,j):\ i<j,\ i,j \in \N \right\}$ and
\[
  \Omega^0=\{ 0,1 \}^{\N^2_<}.
\]
We also define the product $\sigma$-field $\FF^0=2^{\Omega^0}$ and the product measure
\[
  \p_{\bx,t}^0=\bigotimes_{ i<j } \p_{i,j},
\]
where $\p_{i,j}$ is the law of a Bernoulli random variable with success probability $\p_{i,j}\{1\}=\p\left\{ \bA_{i,j}\leq x_ix_jt \right\}$.
If $i>j$ we set $\omega_{i,j}:=\omega_{j,i}$, and we also set $\omega_{i,i}:=1$ for all $i\in \N$.
Elementary events from $\Omega^0$ will specify a family of open edges in $G_t(\bx;\bA)$.
More precisely, 
given $\omega=(\omega_{i,j})_{i<j} \in \Omega^0$,
a pair of vertices $\{i,j\}$ is connected in $G_t(\bx;\bA)(\omega)$ by an edge if and only if $\omega_{i,j}=1$. 
In other words, $\p_{\bx,t}^0$ is an ``inhomogeneous percolation process on the complete infinite graph $(\N,\{\{i,j\}:\ i,j\in \N\})$'' (we include the loops connecting each $i$ to itself on purpose). It should be clear (though not important for the sequel) that the law of thus obtained random graph $G_t(\bx;\bA)$ is the same (modulo loops $\{i,i\}$) as the law of $G_t(\bx;\bA,R^*)$ constructed in~\cite{Konarovskyi:MC:2020}. In particular, the ordered masses of the connected components of 
$G_t(\bx;\bA)$ evolve in $t$ as the \MC\ started from $\ord(\bx)\in \cvd$. \\
For two $i,j\in \N$ we write $\{i \leftrightarrow j\} = \{\{i,j\} \mbox{ is an edge of } G_t(\bx;\bA)\}$ and we  may also write it as $\{\{i,j\} \mbox{ is open}\}$.
% According to our identification, $i\leftrightarrow j$, $i<j$, provided $\omega_{i,j}=1$. 
We also write $\{i\sim j\}$ for the event that $i$ and $j$ belong to the same connected component of the graph $G_t(x;\bA)$. Then we have, $\omega$-by-$\omega$, that $i\sim j$ if and only if there exists a finite path of edges 
\[
  i=i_0\leftrightarrow i_1 \leftrightarrow \dots \leftrightarrow i_l=j.
\]
%connecting $i$ and $j$. 
%We henceforth write $\p\left\{ i \sim j \right\}=\p_{\bx,t}^0\left\{ i \sim j \right\}.$
As already argued, we can write
%Then one can trivially recognize 
\begin{equation}
\label{E:starpage2}
  \p\left( i \sim j \right)=\p_{\bx,t}^0\left( i \sim j \right),
\end{equation}
where $(\Omega, \FF,\p)$ is the underlying probability space and $G_t(\bx;\bA)$ is the above constructed random graph with vertices 
$\N$ and edges
$\{\{i,j\}\in \N^2: i \leftrightarrow j\}$.

\subsection{Disjoint occurence}%
\label{sub:disjoint_occurence}
Our argument partly relies on {\em disjoint occurrence}. We follow the notation from ~\cite{Arratia:2018}, since they work on infinite product spaces. We will use an analog of the van den Berg-Kesten inequality~\cite{vandenBerg:1985}, and also recall that the theorem cited from~\cite{Arratia:2018} is an analog of Reimer's theorem~\cite{Reimer:2000}.  
Given a finite family of events $A_k$, $k \in [n]$, from $\FF^0$ we define the event 
\[
\bsq\limits_{k=1}^n A_k=\{A_k,\, k \in [n], \mbox{ jointly occur for disjoint reasons}\}. 
\]
Readers familiar with percolation can skip the next paragraph and continue reading either at the statement of Lemma \ref{lem_rkb_inequality} or the start at Section~\ref{sec:some_auxiliary_statements}. 

Let for $\omega \in \Omega^0$ and $K \subset \N^2_<$
\[
  \Cyl(K,\omega):=\left\{ \bar{\omega}:\ \bar{\omega}_{i,j} =\omega_{i,j},\ (i,j) \in K\right\}
\]
be the {\it thin cylinder} specified through $K$.
Then the event
\[
  [A]_{K}:=\left\{ \omega:\ \Cyl(K,\omega) \subset A \right\}
\]
is the largest cylinder set contained in $A$, such that it is {\em free in the directions indexed by $K^c$}. Define
\[
  \bsq\limits_{k=1}^n A_k=A_1\sq\dots \sq A_n:=\bigcup_{ J_1,\dots,J_n } [A_1]_{J_1}\cap \dots \cap [A_n]_{J_n},
\]
where the union is taken over finite disjoint subsets $J_k$, $k \in [n]$, of $\N^2_<$.

Let $i_k,j_k \in \N$ and $i_k \not= j_k$, $k \in [n]$. Then we have clearly
\[
  \bsq_{k=1}^n\{ i_k\sim j_k \}=\left\{ 
  	i_k\sim j_k,\ k \in [n],\ \mbox{via mutually disjoint paths}
  \right\}.
\]

The following lemma  follows directly from Theorem~11~\cite{Arratia:2018}, but since the events in question are simple (and monotone increasing in $t$) this could be derived directly in a manner analogous to \cite{vandenBerg:1985}.

\begin{lemma} %RKB inequality
  \label{lem_rkb_inequality}
  For any $i_k,j_k \in \N$ and $i_k \not= j_k$, $k \in [n]$, we have
  \[
    \p_{\bx,t}^0\left( \bsq_{k=1}^n \{i_k\sim j_k\} \right)\leq \prod_{ k=1 }^{ n } \p\left(i_k\sim j_k \right).
  \]
\end{lemma}

{\bf A warning about notation.} We shall denote by $\Lambda_n$ the set of all bijections $\sigma:[n]\to[n]$, since the symbols $S$ and $S_n$, typically used to denote the symmetric group, have been already reserved.
%following the \MC\  tradition, this symbol is already reserved here for the (infinite) sum of squared component masses process. 

%{\bf The paper is organized as follows.} Theorem \ref{the_finiteness_of_fourth_moment}, stated and proved in Section~\ref{the_finiteness_of_fourth_moment}, is the central result of the paper.
%Section~\ref{sec:some_auxiliary_statements} contains some general estimates, upper bounds for the probabilities of inter-connections, which could be of independent interest.
%Section~\ref{sec:conseq_excursions} is the study of finiteness of the second moment of $\|X(\cdot)\|$ in the setting where $X$ is any extreme eternal \MC.  Due to the well-known correspondence between these entrance laws and excursion lengths of a certain class of L\'evy-type processes, these novel results are also stated in the excursion context.

{\bf Structure of the paper.} The remainder of the paper is organized as follows:
Section~\ref{sec:some_auxiliary_statements} is devoted to some general estimates (the upper bounds for the probabilities of inter-connections for the graphs introduced in Section \ref{sub:graph_construction} could be of independent interest),
Theorem \ref{the_finiteness_of_fourth_moment} is proved in Section~\ref{sec:finiteness_of_first_moment},
%and is the central result of the paper. 
and the proof of Theorem~\ref{the_finiteness_of_the_second_moments_of_emc} is given in Section~\ref{sec:conseq_excursions}.

\section{Some auxiliary statements}%
\label{sec:some_auxiliary_statements}
We work on $(\Omega, \FF,\p)$ and with the random graph constructed in (\ref{E:starpage2}).
Let us recall the following easy lemma, known already to Aldous and Limic
(see p.~46 in \cite{Aldous:1998} or expression (2.2) on p.~10 in \cite{Limic:1998:thesis}, or 
 for example~\cite{Konarovskyi:MC:2020} for details).

\begin{lemma} %connection of two paths
  \label{lem_connection_of_two_paths}
  For every $\bx=(x_k)_{k\geq 1} \in l^2$, $t \in (0,1/\|\bx\|^2)$ and $i \not= j$
  \[
    \p\left( i\sim j \right)\leq \frac{ x_ix_jt }{ 1-t \|\bx\|^2 }.
  \]
\end{lemma}

The goal of this section is to obtain analogous estimates for the probability of connection for $n$-tuples of vertices. 

\begin{proposition} %estimate for connected componnents
  \label{pro_estimate_for_connected_componnents}
  For every $n \in \N$ there exists a constant $C_n$ such that for every $\bx=(x_k)_{k\geq 1} \in l^2$ and $t \in (0,1/\|\bx\|^2)$
  \begin{equation} %main estimate of the probability
  \label{equ_main_estimate_of_the_probability}
    \p\left( i_1\sim i_2 \sim \dots\sim i_n\right)\leq C_n\frac{ x_{i_1}x_{i_2}\dots x_{i_n}t^{n/2} }{ \left(1-t \|\bx\|^2\right)^{2n-3} }
  \end{equation}
where $i_k$, $k \in [n]$, is an arbitrary collection of $n$ distinct indices (natural numbers).
\end{proposition}

\begin{proof}
  Let $i_1,\dots,i_n$ be distinct natural numbers. 
We will consider $\left\{ i_1\sim\dots\sim i_n \right\}$ as an event on the probability space $(\Omega^0,\FF^0,\p^0_{\bx,t})$ (see also \eqref{E:starpage2}). We remark that $\left\{ i_1\sim\dots\sim i_n \right\}$ happens if and only if there exists a minimal spanning tree containing the vertices $i_1,\dots,i_n$.
 More precisely, the event 
 $\left\{ i_1\sim\dots\sim i_n \right\}$
 coincides with the event that
 there exists a connected (random) graph $G_{i_1,\dots,i_n}\subset G_t(\bx;\bA)$ 
 without cycles, such that $\{i_k:k \in [n]\}$ is contained in its vertices,
furthermore the leaves of $G_{i_1,\dots,i_n}$ are WLOG contained in 
 $\{i_k:k \in [n]\}$
 and a deletion of any interior (non-leaf) vertex $j \not\in \{i_k: k \in [n]\}$ together with the corresponding
 incident edges would make $G_{i_1,\dots,i_n}$ a disconnected graph (in this case a forest).
 In the rest of this argument we shall write $j \in G_{i_1,\dots,i_k}$ to mean that $j$ is a vertex of $G_{i_1,\dots,i_n}$. 
 Minimal spanning trees may not be unique, but here we only care about existence.

  We will prove the proposition using mathematical induction. Inequality~\eqref{equ_main_estimate_of_the_probability} for $n=2$ is the statement of Lemma~\ref{lem_connection_of_two_paths}. 
  The induction hypothesis is~\eqref{equ_main_estimate_of_the_probability} for all $n=2,\dots,N$ and the step is to prove the same for $n=N+1$.

Now note that on the event that $G_{i_1,\dots,i_{N+1}}$ exists, it must be that either $i_{N+1}$ is one of its leaves or it is one of its interior vertices.
   Setting $\tilde{\N}:=\N\setminus\{i_{N+1}\}$, we can therefore estimate
  \begin{align}
    \p_{\bx,t}^0\big(i_1&\sim\dots\sim i_{N+1} \big)= \p_{\bx,t}^0\left( \exists G_{i_1,\dots,i_{N+1}} \right)\\
&\leq  \p_{\bx,t}^0\left( \bigcup_{ j \in \tilde{\N}} \left\{ i_{N+1}\sim j \right\}\sq\left\{ \exists G_{i_1,\dots,i_N} \ni j\} \right\}\right) \nonumber \\
&+\p_{\bx,t}^0\left( \bigcup_{\sigma \in \Lambda_N} \bigcup_{l \in [N-1]}\ \ \left\{ \exists G_{i_{\sigma(1)},\dots,i_{\sigma(l)},i_{N+1}}\right\} \sq\left\{ \exists G_{i_{\sigma(l+1)},\dots,i_{\sigma(N)},i_{N+1}} \right\} \right) \nonumber \\
&\leq \sum_{ j \in \tilde{\N}} \p_{\bx,t}^0\left( \left\{ i_{N+1}\sim j \right\}\sq\left\{ \exists G_{i_1,\dots,i_N} \ni j\} \right\}\right) \nonumber \\
&+\sum_{ \sigma \in \Lambda_N } \sum_{ l=1 }^{ N-1 } \p_{\bx,t}^0\left( \left\{ \exists G_{i_{\sigma(1)},\dots,i_{\sigma(l)},i_{N+1}}\right\} \sq\left\{ \exists G_{i_{\sigma(l+1)},\dots,i_{\sigma(N)},i_{N+1}} \right\} \right). \label{E:starpage4}
  \end{align}
  \begin{figure}[H]
      \centering
      \includegraphics[width=70mm]{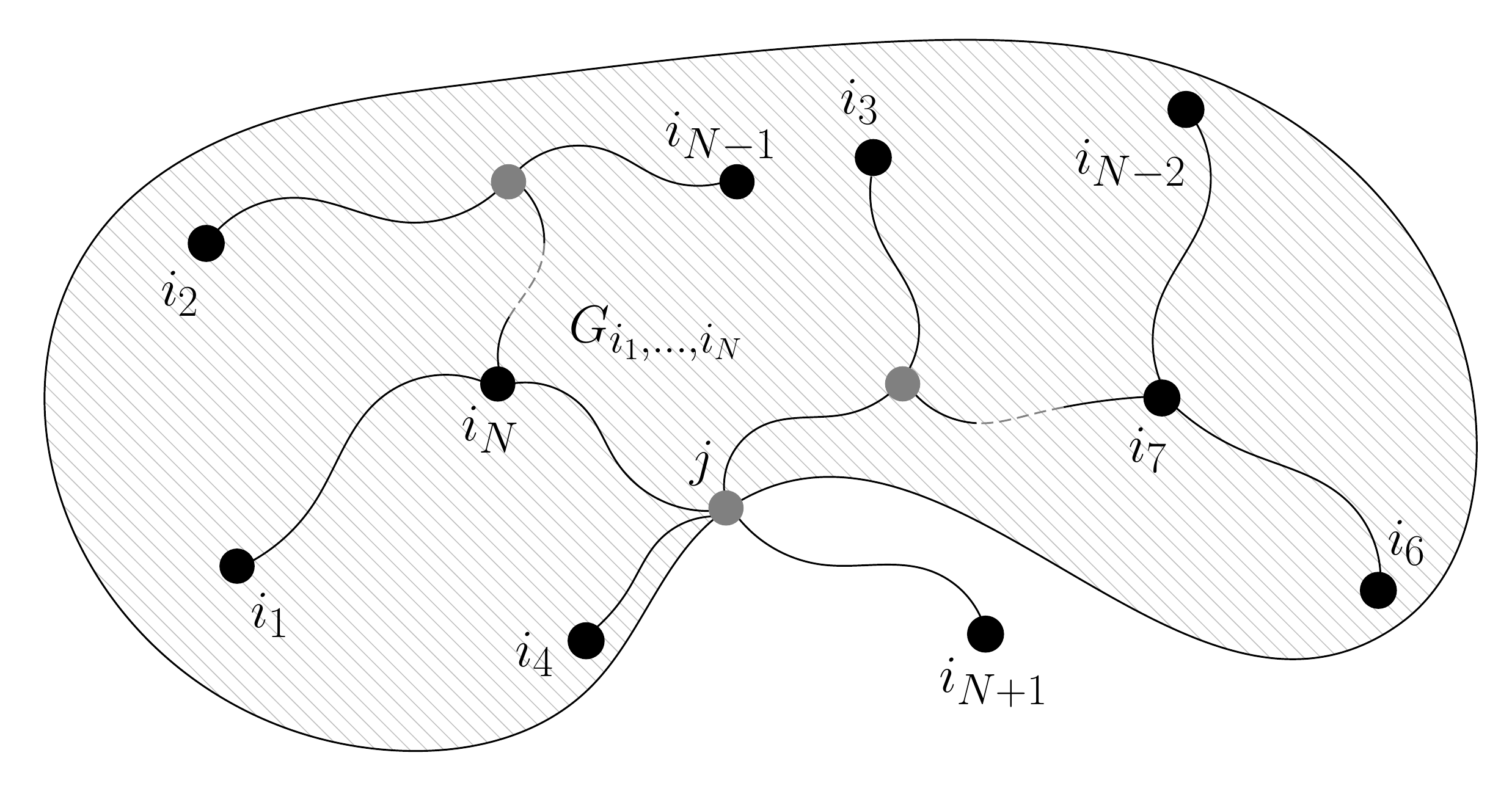}
          \includegraphics[width=70mm]{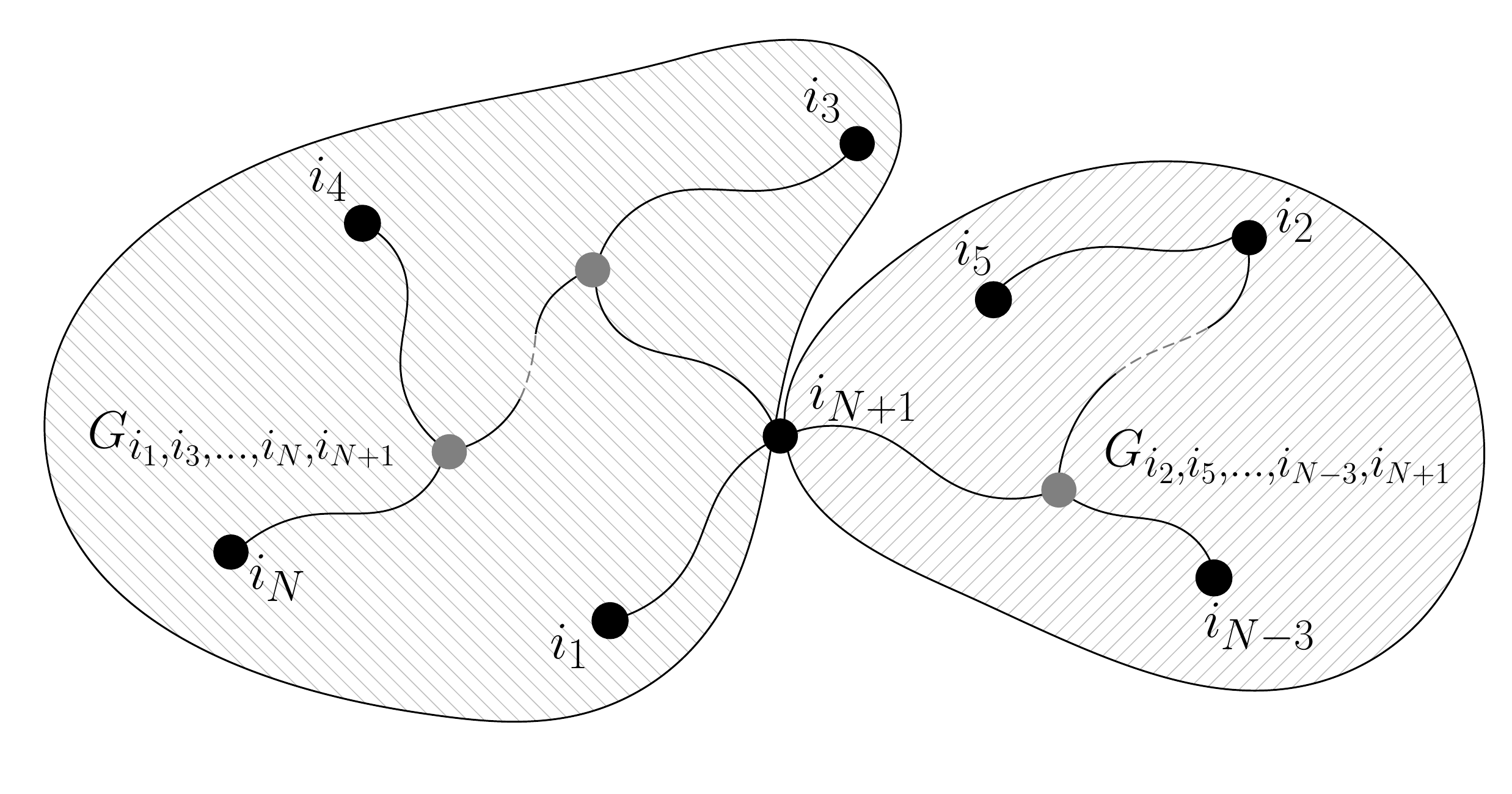}
      \caption{This illustrates the cases $i_{N+1}$ is a leaf (on the left), and $i_{N+1}$ is an interior vertex (on the right).}
  \end{figure}
  We next estimate each term on the right hand side of \eqref{E:starpage4}, starting with the terms at the end, and then moving onto the terms in the second to last line.
 Due to Lemma~\ref{lem_rkb_inequality} and the induction hypothesis, one has 
  \begin{align}
  \nonumber
    \p_{\bx,t}^0\big( \big\{ \exists &G_{i_{\sigma(1)},\dots,i_{\sigma(l)},i_{N+1}}\big\} \sq\left\{ \exists G_{i_{\sigma(l+1)},\dots,i_{\sigma(N)},i_{N+1}} \right\} \big)\\
  \nonumber
    &\leq \p_{\bx,t}^0\left( \exists G_{i_{\sigma(1)},\dots,i_{\sigma(l)},i_{N+1}} \right)\p_{\bx,t}^0\left( \exists G_{i_{\sigma(l+1)},\dots,i_{\sigma(N)},i_{N+1}} \right)=\\
\nonumber
    &=\p_{\bx,t}^0\left( i_{\sigma(1)}\sim\dots\sim i_{\sigma(l)}\sim i_{N+1} \right)\p_{\bx,t}^0\left( i_{\sigma(l+1)}\sim\dots\sim i_{\sigma(N)}\sim i_{N+1} \right)\\
\nonumber
    &\leq C_{l+1}\frac{ x_{i_{\sigma(1)}}\dots x_{i_{\sigma(l)}}x_{i_{N+1}}t^{\frac{ l+1 }{ 2 }} }{ \left(1-t \|\bx\|^2\right)^{2(l+1)-3} }\cdot C_{N-l+1}\frac{ x_{i_{\sigma(l+1)}}\dots x_{i_{\sigma(N)}}x_{i_{N+1}}t^{\frac{ N-l+1 }{ 2 }} }{ \left(1-t \|\bx\|^2\right)^{2(N-l+1)-3} }\\
\nonumber
    &= C_{l+1}C_{N-l+1}\frac{ x_{i_1}\dots x_{i_{N+1}}t^{\frac{ N+1 }{ 2 }}}{ (1-t \|\bx\|^2)^{2N-2}}\cdot x_{i_{N+1}}t^{ \frac{1}{ 2 }}\\
\label{E:inducstep1}    
    &\leq C_{l+1}C_{N-l+1}\frac{ x_{i_1}\dots x_{i_{N+1}}t^{\frac{ N+1 }{ 2 }}}{ (1-t \|\bx\|^2)^{2N-1}},
  \end{align}

\vspace{-0.3cm}
\noindent
  where in the final step we used the facts that $t^{ \frac{1}{ 2 }}x_{i_{N+1}}\leq t^{ \frac{1}{ 2 }}\|\bx\|<1$ and $1-t \| \bx\|^2\leq 1$. 
  
Now let us denote $I\equiv I_N:=\{i_1,\dots,i_N\}$ and let $ I^c=\tilde{\N}\setminus I$.
Let us first assume that $j \in I$.
Then, similarly to the just made computation, we have
  \begin{align}
  \nonumber
    \p_{\bx,t}^0\big( \left\{ i_{N+1}\sim j \right\}&\sq\left\{ \exists G_{i_1,\dots,i_N} \ni j \right\}\big)= \p_{\bx,t}^0\left( \left\{ i_{N+1}\sim j \right\}\sq\left\{ \exists G_{i_1,\dots,i_N} \right\}\right)\\
\nonumber
  &\leq \p_{\bx,t}^0\left( i_{N+1}\sim j \right)\p_{\bx,t}^0\left(\exists G_{i_1,\dots,i_N}\right)\\
\nonumber
  &= \p_{\bx,t}^0\left( i_{N+1}\sim j \right)\p_{\bx,t}^0\left(i_1\sim\dots\sim i_N\right)\\
\label{E:inducstep2}    
  &\leq \frac{ x_{i_{N+1}}x_{j}t }{ 1-t \|\bx\|^2 }\cdot C_N \frac{ x_{i_1}\dots x_{i_N}t^{ \frac{N}{ 2 }} }{ \left(1-t \|\bx\|^2\right)^{2N-3} }\leq C_N \frac{ x_{i_1}\dots x_{i_{N+1}}t^{\frac{ N+1 }{ 2 }} }{ \left(1-t \|\bx\|^2\right)^{2N-1} },
  \end{align}
  where we used again the estimates $x_jt^{1/2}<1$ and $1-t \| \bx\|^2\leq 1$.
  Next, let us assume that $j \in I^c$. Then $j$ is necessarily an interior vertex of the minimal spanning tree $G_{i_1,\dots i_N}$.
  In particular, $G_{i_1,\dots i_N}$ is a union of two 
  minimal spanning trees $\G_{i_{\sigma(1)},\dots,i_{\sigma(l)},j}$ and $G_{i_{\sigma(l+1)},\dots,i_{\sigma(N)},j}$, for some $\sigma \in \Lambda_{N}$ and $l \in [N-1]$,
  which have %(pair-wise)
 disjoint edge sets, and their only vertex in common is $j$. 
  Therefore, using the induction hypothesis and Lemma~\ref{lem_rkb_inequality}, we can estimate
  \begin{align*}
    \sum_{ j \in I^c }\p_{\bx,t}^0&\left( \left\{ i_{N+1}\sim j \right\}\sq\left\{ \exists G_{i_1,\dots,i_N} \ni j \right\}\right)\leq  \sum_{ j \in I^c }\p_{\bx,t}^0\left( i_{N+1}\sim j \right)\p_{\bx,t}^0\left(\exists G_{i_1,\dots,i_N} \ni j\right)\\
&\leq \sum_{ j \in I^c }\sum_{ \sigma \in \Lambda_N } \sum_{ l=1 }^{ N-1 } \p_{\bx,t}^0\left( i_{N+1}\sim j \right)\p_{\bx,t}^0\left(\left\{\exists G_{i_{\sigma(1)},\dots,i_{\sigma(l)},j}\right\}\sq\left\{\exists G_{i_{\sigma(l+1)},\dots,i_{\sigma(N)},j}\right\}\right)\\
&\leq  \sum_{ j \in I^c }\sum_{ \sigma \in \Lambda_N } \sum_{ l=1 }^{ N-1 } \p_{\bx,t}^0\left( i_{N+1}\sim j \right)\p_{\bx,t}^0\left(\exists G_{i_{\sigma(1)},\dots,i_{\sigma(l)},j}\right)\p_{\bx,t}^0\left(\exists G_{i_{\sigma(l+1)},\dots,i_{\sigma(N)},j}\right).
  \end{align*}
  \begin{figure}[H]
      \centering
      \includegraphics[width=70mm]{p_01.pdf}
          \includegraphics[width=70mm]{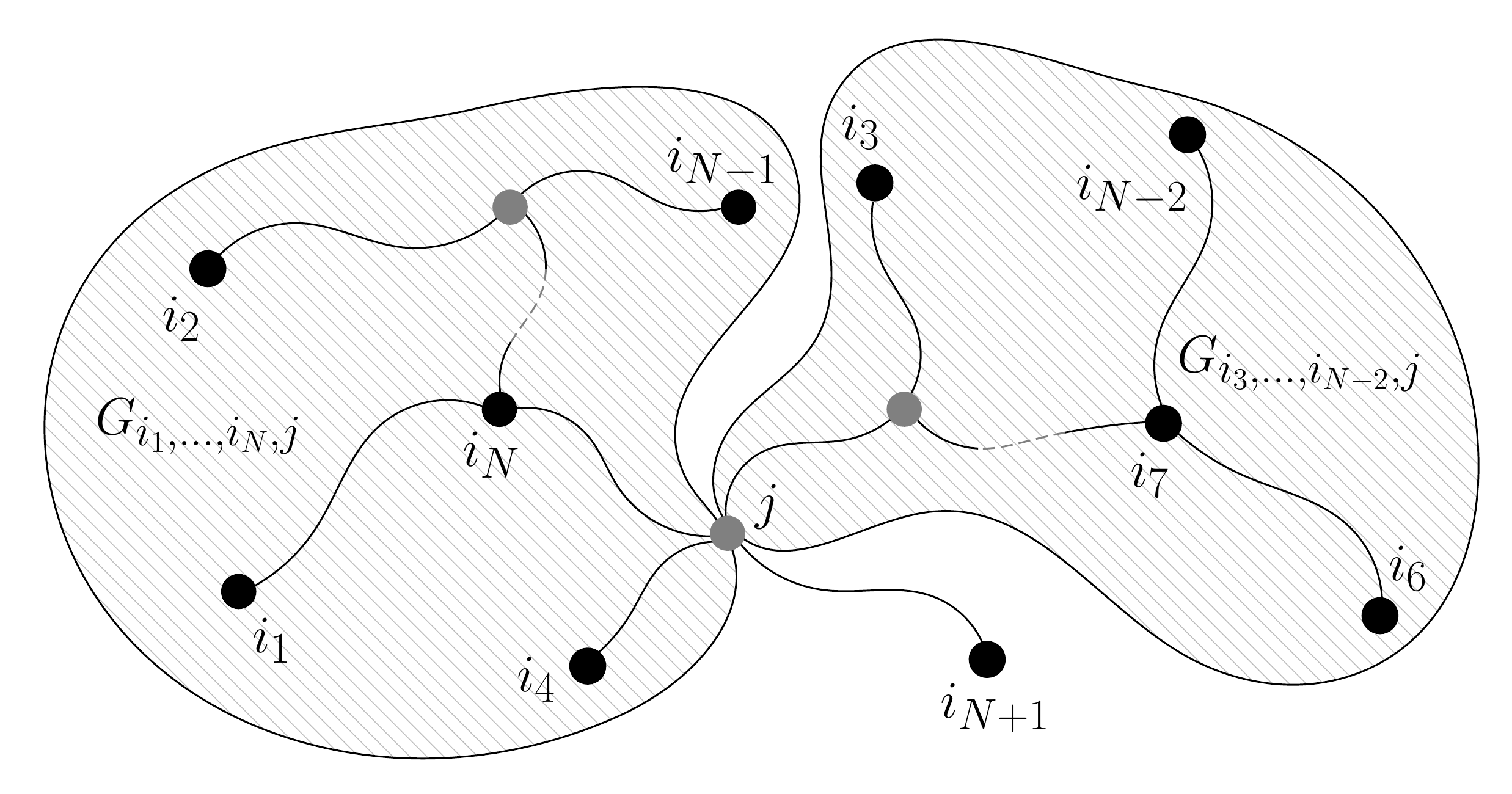}
      \caption{This figure illustrates the cases $j\in I$ (to the left) and  $j\in I^c$ (to the right).}
  \end{figure}

We can now use the induction hypothesis to bound the right hand side of the last expression by
\begin{align}
\nonumber
\sum_{ j \in I^c }\sum_{ \sigma \in \Lambda_N }& \sum_{ l=1 }^{ N-1 } \frac{ x_{i_{N+1}}x_jt }{ 1-t \|\bx\|^2 }\cdot C_{l+1}\frac{ x_{i_{\sigma(1)}}\dots x_{i_{\sigma(l)}}x_j t^{\frac{ l+1 }{ 2 }}}{ (1-t \|\bx\|^2)^{2(l+1)-3} }\cdot C_{N-l+1}\frac{ x_{i_{\sigma(l+1)}}\dots x_{i_N}x_j t^{\frac{ N-l+1 }{ 2 }} }{ (1-t \|\bx\|^2)^{2(N-l+1)-3} }\\
\nonumber
&=\sum_{ \sigma \in \Lambda_N } \sum_{ l=1 }^{ n-1 } C_{l+1}C_{N-l+1}\frac{ x_{i_1}\dots x_{i_{N+1}} t^{\frac{ N+4 }{ 2 }}}{ (1-t \|\bx\|^2)^{2N-1} }\left(\sum_{ j \in I^c }x_j^{3}\right)\\
\nonumber
&\leq N! \sum_{ l=1 }^{ N-1 } C_{l+1}C_{N-l+1}\frac{ x_{i_1}\dots x_{i_{N+1}} t^{\frac{ N+1 }{ 2 }}}{ (1-t \|\bx\|^2)^{2N-1} }\cdot\|\bx\|^3t^{ \frac{3}{ 2 }}\\
&\leq N! \sum_{ l=1 }^{ N-1 } C_{l+1}C_{N-l+1}\frac{ x_{i_1}\dots x_{i_{N+1}} t^{\frac{ N+1 }{ 2 }}}{ (1-t \|\bx\|^2)^{2N-1} }.
\label{E:inducstep3}
  \end{align}
In the above computation we again used the estimate $\|\bx\|t^{1/2}<1$.
Combining estimate \eqref{E:starpage4} with (\ref{E:inducstep1}--\ref{E:inducstep3}) completes the induction step with $C_{n+1} = 2n! \sum_{ l=1 }^{ n-1 } C_{l+1}C_{n-l+1} + n C_n$, and therefore the whole argument.
\end{proof}

\section{\texorpdfstring{$\|\bX(t)\|$}{X(t)} has all moments at all times}%
\label{sec:finiteness_of_first_moment}

Recall that $\bX$ is a multiplicative coalescent starting from $\bx \in \cvd$, and that $G_t(\bx;\bA)$ is a graphical representation of $\bX(t)$, as described in Section~\ref{sub:graph_construction}.
The main goal of this section is to prove Theorem~\ref{the_finiteness_of_fourth_moment}. We first show that the $n$-th moment of $\|\bX(t)\|$ is finite for small $t$, and then we extend this result to all $t\geq 0$.
\subsection{Argument for small times}\label{sub:argument_for_small_times}
We still assume that $\bX$ is started at time $0$ from initial configuration $\bx$.
\begin{lemma} %finitenes of fourth norm
  \label{lem_finitenes_of_expected_S_n}
  For each $n \geq 2$ there exists a constant $D_n>0$ such that for each $\bx \in l^2$ and $t \in (0,1/\|\bx\|^2)$ we have
  \[
    \E\sum_{ k=1 }^{ \infty } X_k^n(t)< \frac{ D_n\|\bx\|^n }{ \left(1-t \|\bx\|^2\right)^{2n-3} }\ .
  \]
\end{lemma}

\begin{proof} %
 First note that all $X_k(t)$ are non-negative random variables, so that due to the monotone convergence theorem the expectation and the summation can be exchanged. We will apply the Fubini-Tonelli theorem after making the following observations.
  
  At time $t$, the largest component (with mass $X_1(t)$) is formed from individual (original) blocks with indices in a random set denoted by $I_1 \subset \N$, the second largest component (with mass $X_2(t)$) is formed from original blocks with indices in $I_2 \subset \N$, and similarly the $k$th largest component (with mass $X_k(t)$) is formed from individual (original) blocks with indices $I_k \subset \N$.
  We know that $\N$ equals the disjoint union of $I_k$, $k\in \N$.
  Next observe that 
  $$
  X_k^n(t) = \left(\sum_{j \in I_k} x_j\right)^n = \sum_{j_1\in I_k} \sum_{j_2\in I_k} \ldots \sum_{j_n\in I_k}
  x_{j_1}x_{j_2}\cdots x_{j_n},
  $$
  so that
  \begin{equation}
  \sum_{ k=1 }^{ \infty } X_k^n(t) = \sum_{i_1=1}^\infty \sum_{i_2=1}^\infty \ldots \sum_{i_n=1}^\infty x_{i_1}x_{i_2}\cdots x_{i_n} \I_{\{i_1 \sim i_2 \sim \ldots \sim i_n\}}.
  \label{E:sumXkton}    
  \end{equation}
  
  Out of convenience we apply here a natural convention that $i\sim i$ for each $i\in \N$, as indicated in Section~\ref{sub:graph_construction}. 
  
  The (finite) family of all partitions $\pi=\{\pi_1,\dots,\pi_p\}$ of $[n]$ will be denoted by $\Pi_n$. 
  If $\pi\in \Pi_n$ we will write $p$ or $|\pi|$ for a number of distinct sets
  (or equivalently, the number of equivalence classes) in $\pi$. 
  Similarly, if $e$ is an equivalence class of $\pi$ then $|e|$  denotes the number of distinct elements in $e$.
  Each equivalence class $e$ is a subset of $[n]$ and therefore it has its minimal element $\min(e)$.
  It is convenient to increasingly order the equivalence classes in $\pi$ with respect to their minimal elements.
  Let $o(\pi):[n]\to[|\pi|]$ be the map which assigns to each $i$ the rank of its equivalence class with respect to the just defined ordering.
  In particular, $o(\pi)(1)$ is always equal to $1$, $o(\pi)(k)=2$ for minimal $k=k_1$ such that $k \not\sim_\pi 1$, $o(\pi)(k)=3$ for minimal $k=k_2$ such that $k \not\sim_\pi 1$ and  $k \not\sim_\pi k_1$, and so on.  
  Note that $\pi$ can be completely recovered from $o(\pi)$.
  
  Each $n$-tuple $(i_1,\ldots,i_n)$, where coordinates are in $\N$
  is equivalent to a function from $[n]$ to $\N$, and each such function $i$
  can be bijectively mapped into a labelled partition of $[n]$, where
  $a$ is related to $b$ iff $i_a=i_b$, and the label of each equivalence class is
  precisely the value (natural number) which $i_\cdot$ takes on any of its elements.
  %Let $\Pi_n^*$ denote this family of all labelled partitions, and 
  Let $\pi[(i_1,\ldots,i_n)]$ be the labelled partition which uniquely corresponds to $(i_1,\ldots,i_n)$.
The reader should note that this newly defined partition structure is completely different from (unrelated to) the random connectivity relation induced by the random graph. 
  Here and elsewhere in the paper we reserve the symbol $\sim$ to denote the latter relation.

  With this correspondence in mind, note that the $n$-fold summation 
  $\sum_{ i_1,\dots,i_n=1 }^{ \infty } g(i_1, \ldots, i_n)$
  can be rewritten as 
  $\sum_{p=1}^n \sum_{ \pi\in \Pi_n: |\pi|=p } \sum_{i_1,i_2,\ldots,i_p=1 \text{ distinct}}^\infty g(i_{o(\pi)(1)},\ldots,i_{o(\pi)(n)})$.
In particular, for any fixed $\pi \in \Pi_n$ and any $n$-tuple $(i_1,\ldots, i_n)$ such that $\pi=\pi[(i_1,\ldots,i_n)]$ %we have that 
  $$\p\left( i_1 \sim \dots \sim i_n \right) =
  \p\left( i_{o(\pi)(1)} \sim \dots \sim i_{o(\pi)(n)} \right) = 
  \p\left( i_1 \sim i_{k_1} \dots \sim l_{k_p} \right)\!,$$
  where $k_j$ is precisely the minimal element of the $j$th equivalence class in $\pi$.

  Using \eqref{E:sumXkton}, the just given reasoning, and Proposition~\ref{pro_estimate_for_connected_componnents}  
  %and the fact $t \|\bx\|^2<1$, 
  we can now estimate 
  \begin{align*}
   \E \sum_{ k=1 }^{ \infty }  X_k^n(t)&= \sum_{ i_1,\dots,i_n=1 }^{ \infty } x_{i_1}\dots x_{i_n}\p\left( i_1 \sim \dots \sim i_n \right)\\
    &= \sum_{p=1}^n \sum_{ \pi \in \Pi_n: |\pi|=p } \ \sum_{ i_1,\dots, i_p \text{ distinct}} x_{i_{o(\pi)(1)}}\cdots x_{i_{o(\pi)(n)}} \p\left( i_{o(\pi)(1)}  \sim\dots\sim i_{o(\pi)(n)} \right)\\
    &= \sum_{p=1}^n \sum_{ \pi \in \Pi_n: |\pi|=p } \ \sum_{ i_1,\dots, i_p \text{ distinct}} x_{i_1}^{|e_1|}\cdots x_{i_p}^{|e_p|}\p\left( i_1 \sim\dots\sim i_p \right)\\
    &\leq \sum_{p=1}^n \sum_{ \pi \in \Pi_n: |\pi|=p } \,
    \sum_{ i_1,\dots, i_p \text{ distinct}} x_{i_1}^{|e_1|}\cdots x_{i_p}^{|e_p|}\frac{ C_px_{i_1}\dots x_{i_p}t^{p/2} }{ \left(1-t \|\bx\|^2\right)^{2p-3} },
  \end{align*}
 where $e_1, \ldots, e_p$ are the equivalence classes of $\pi$, ordered as explained above.
If we replace the interior sum (over distinct $p$-tuples) by the sum over all $p$-tuples, and again recall that $0\leq t\|x\|^2< 1$ and that 
$\sum y_k^{|e|+1}\leq \|y\|^{|e|+1}$, we get a further upper bound
  \begin{align*}
     \sum_{p=1}^n \sum_{ \pi \in \Pi_n:|\pi|=p }& \frac{C_pt^{p/2} }{ \left(1-t \|\bx\|^2\right)^{2p-3} }\sum_{ i_1, \dots, i_p=1 }^{\infty} x_{i_1}^{|e_1|+1}\cdots x_{i_p}^{|e_p|+1}\\
    &\leq \sum_{p=1}^n \sum_{ \pi \in \Pi_n: |\pi|=p } \frac{C_pt^{p/2} }{ \left(1-t \|\bx\|^2\right)^{2p-3} } \|\bx\|^{|e_1|+1}\cdots \|\bx\|^{|e_p|+1}\\
    &\leq \sum_{p=1}^n \sum_{ \pi \in \Pi_n: |\pi|=p } \frac{C_pt^{p/2}\|\bx\|^{n+p} }{ \left(1-t \|\bx\|^2\right)^{2p-3} }\leq \sum_{p=1}^n  \sum_{ \pi \in \Pi_n: |\pi|=p } \frac{C_p\|\bx\|^n }{ \left(1-t \|\bx\|^2\right)^{2n-3} }\\
    &= \frac{\|\bx\|^n }{ \left(1-t \|\bx\|^2\right)^{2n-3}} 
    \sum_{\pi \in \Pi_n} C_{|\pi|},
  \end{align*}
 and this is the stated claim with $D_n:= \sum_{\pi \in \Pi_n} C_{|\pi|}$.
\end{proof}

As already mentioned, the \MC\ $(\bX(t),\, t\geq 0)$ is a Markov process taking values in $\cvd$. Applying its generator $\Gamma$ to $g(\bX(t))$, where $g$ is an arbitrary function from $\cvd$ to $\R$, one can conclude that the process 
\begin{equation} %martingale
  \label{equ_martingale}
  M_g(t):=g(\bX(t))-\int_{ 0 }^{ t }  \Gamma g(\bX(r)) dr, \quad t\geq 0, 
\end{equation}
is a local $(\FF_t)$-martingale (see also identity (66) in~\cite{Aldous:1998}). Here $\FF_t=\sigma(\bX(r),\ r\leq t)$, $t\geq 0$, and the generator $\Gamma$ of $\bX(t)$, $t\geq 0$, is defined as
\[
  \Gamma g(\bx)=\sum_{ i=1 }^{ \infty } \sum_{ j=i+1 }^{ \infty } x_ix_j(g(\bx^{i,j})-g(\bx)),
\]
where $\bx^{i,j}$ is the configuration obtained from $\bx$ by merging the $i$-th and $j$-th clusters, or equivalently (assuming that $i<j-1$, the other cases can be written similarly) $\bx^{i,j}=(x_1,\dots,x_{l-1},x_i+x_j,x_l,\dots,x_{i-1},x_{i+1},\dots,x_{j-1},x_{j+1},\dots)$ for some $l$ such that $x_{l-1} \leq x_i+x_j\leq x_l$.  We will use $\Gamma$ and \eqref{equ_martingale} in order to show the finiteness of the $n$-th moment of the multiplicative coalescent at small times. 

We first prove an auxiliary statement, which does not require $\bX$ to have \MC\ law (it is sufficient for the process to be \cdl\!\!), probably known in the literature, but we were unable to find a precise reference. 
  Let $g,f_1,f_2:\cvd \to [0,\infty)$ be measurable functions such that $g$ is continuous and suppose that
  \[
    M(t)=g(\bX(t))-\int_{ 0 }^{ t } \left( f_2(\bX(r))-f_1(\bX(r)) \right)dr, \quad t\geq 0, 
  \]
 is a local $(\FF_t)$-martingale. 
Define 
  \[
    \tau_n:=\inf\left\{ t:\ 
      \begin{array}{l} %
	\max\big\{|M(t)|,|M(t-)|,g(\bX(t)),g(\bX(t-))\big\}\geq n\\ 
         \mbox{or}\ \ \int_{ 0 }^{ t } f_1(\bX(r))dr\geq n
      \end{array}  \right\}, \quad n\geq 1.
  \]
Continuity hypothesis on $g$ assures that $\tau_n$ is an $(\FF_t)$-stopping time.
More precisely, since $(M(t),\, t\geq 0)$ and $(g(\bX(t)),\, t\geq 0)$ are right continuous processes with left limits, $\tau_n$, $n\geq 1$ are $(\FF_t)$-stopping times, by Proposition~2.1.5~(a)~\cite{Ethier:1986}. 
\begin{lemma} %recurrent property
  \label{lem_recurrent_property}
If $\E \sup\limits_{ r \in [0,t] }g(\bX(r))<\infty$ and $\E \int_{ 0 }^{ t } f_1(\bX(r))dr<\infty $ for some $t>0$, then also $\E \int_{ 0 }^{ t } f_2(\bX(r))dr<\infty $.
\end{lemma}

\begin{proof} %
  From the assumptions we can conclude that $\tau_n \nearrow \infty$ a.s.~as $n \to \infty$. 
Note that $M(t\wedge\tau_n)$, $t\geq 0$, is bounded, and therefore, it is an $(\FF_t)$-martingale for every $n\geq 1$. Thus for any given $n\geq 1$
  \[
    \E M(t\wedge \tau_n)=\E g(\bX(t\wedge\tau_n))-\E\int_{ 0 }^{ t\wedge\tau_n }\left( f_2(\bX(r)) - f_1(\bX(r))\right)dr = g(\bx).
  \]
By monotone convergence and Fatou's lemma, one can now estimate
  \begin{align*}
    \E \int_{ 0 }^{ t } f_2(\bX(r))dr\leq & \varliminf_{ n\to\infty }\E\int_{ 0 }^{ t\wedge \tau_n }f_2(\bX(r))dr= \varliminf_{ n\to\infty }\E g(\bX(t\wedge \tau_n))\\
    +&\varliminf_{ n\to\infty }\E \int_{ 0 }^{ t\wedge \tau_n } f_1(\bX(r))dr-g(\bx)\\
    &\leq \E \sup\limits_{ r \in [0,t] }g(\bX(r))+\int_{ 0 }^{ t } f_1(\bX(r))dr-g(\bx)<\infty,	
  \end{align*}
as stated.
\end{proof}

From now on we again assume that $\bX$ is a \MC\ started from $\bx$.
Define functions $s_n:\cvd \to [0,\infty)$ as
$s_n(\bx):=\sum_{ k=1 }^{ \infty } x_k^n$, for each $n\in \N$.
It is easy to see that  $s_n$ is continuous for each $n\geq 2$.
Led by previous \MC\ literature, we denote 
  \[
    S_n(t):=s_n\left(\bX(t)\right)\equiv \sum_{ k=1 }^{ \infty } X_k^n(t), \quad t\geq 0,
  \]
  and $S(t):=S_2(t)=\|\bX(t)\|^2$, $t\geq 0$. 
 \begin{proposition} %finiteness of fourth moment
  \label{pro_finiteness_of_fourth_moment}
  For every $n\geq 1$, $\bx \in l^2$ and $t \in [0,1/\|\bx\|^2)$
  \begin{equation} %expectation of integral of X
  \label{equ_expectation_of_integral_of_x} 
  \E \|\bX(t)\|^n< +\infty.
  \end{equation}
\end{proposition}

\begin{proof} %
  We shall prove the proposition in two steps.
  The goal of {\it step one} is to show that
  \begin{equation} %finiteness of expectation of product
  \label{equ_finiteness_of_expectation_of_product}
  \E \left(S_n(t)S_m(t)\right)<\infty, \quad t \in [0,1/\|\bx\|^2),
  \end{equation}
  for all $n,m\geq 2$. 
  
  We start by computing the value of the generator $\Gamma$ of $\bX(t)$, $t\geq 0$, on  functions $s_n(\bx)$ for $\bx \in \cvd$ and for odd $n=2k+1\geq  3$
  \begin{align*}
    \Gamma s_{2k+1}(\bx)&= \sum_{ i=1 }^{ \infty } \sum_{ j=i+1 }^{ \infty } x_ix_j\left(s_{2k+1}(\bx)+\sum_{ l=1 }^{ 2k } 
    \binom{2k+1}{l} x_i^lx_j^{2k+1-l}-s_{2k+1}(\bx)\right)\\ %C_{2k+1}^l
    &= \sum_{ l=1 }^{ 2k }\binom{2k+1}{l}  \left(\sum_{ i=1 }^{ \infty } \sum_{ j=i+1 }^{ \infty } x_i^{l+1}x_j^{2k-l+2}\right)  \text{ [symmetry about $k$]}\\
    &= \sum_{ l=1 }^{ k } \binom{2k+1}{l} \left( \sum_{ i\not= j } x_i^{l+1}x_j^{2k-l+2} \right) = \text{ [plus/minus diagonal terms]}\\ %C_{2k+1}^l
    &= \sum_{ l=1 }^{ k }\binom{2k+1}{l} \left( s_{l+1}(\bx)s_{2k-l+2}(\bx)-s_{2k+3}(\bx) \right)\\
    &=  \sum_{ l=1 }^{ k } \binom{2k+1}{l} s_{l+1}(\bx)s_{2k-l+2}(\bx)-\sfrac{1}{2}(2^{2k+1}-2)s_{2k+3}(\bx),
  \end{align*}
  where we recognize the final term as $f_2(\bx)-f_1(\bx)$, with both $f_1,f_2$ non-negative.
Therefore,
\[
  S_{2k+1}(t)-\int_{ 0 }^{ t } \left(f_2(\bX(r))-f_1(\bX(r))\right)dr, \quad t\geq 0, 
\]
is a local $(\FF_t)$-martingale. We note that $S_{2k+1}(t)$, $t\geq 0$, is a non-decreasing process. Hence Lemma \ref{lem_finitenes_of_expected_S_n} guarantees
\[
  \E \sup\limits_{ r \in [0,t] }S_{2k+1}(r)\leq \E S_{2k+1}(t)<\infty,
\] 
and also
\[
 \E \int_{ 0 }^{ t } f_1(\bX(r))dr=(2^{2k}-1)\int_{ 0 }^{ t } \E S_{2k+3}(r)dr<\infty  
\]
for every fixed $t \in [0,1/\|\bx\|^2)$.  Due to the above stated continuity of functions $s_n$, $n\geq 2$, all the hypotheses of Lemma~\ref{lem_recurrent_property} are satisfied, yielding
\[
  \int_{ 0 }^{ t } \E S_{l+1}(r)S_{2k-l+2}(r)dr<\infty. 
\]
for all $l \in [k]$ and $t \in [0,1/\|\bx\|^2)$. Using the monotonicity of $S_n(t)$, $t\geq 0$, once again, we derive~\eqref{equ_finiteness_of_expectation_of_product} for all $n,m\geq 2$ such that $n+m\geq 5$ is an odd number.
A similar computation applied to $s_{2k}$ instead of $s_{2k+1}$ yields~\eqref{equ_finiteness_of_expectation_of_product} for all $n,m\geq 2$ and $n+m\geq 4$ an even number.

In {\it step two} we show the following extension: for every $k\geq 1$
%\begin{equation} %
%  \label{equ_finiteness_of_the_expectation_for_S2n_Sm}
%  \E S^{k-1}(t)S_m(t)S_l(t)<\infty, \quad t \in [0,1/\|\bx\|^2),\ \ m,l\geq 2,
%\end{equation}
\begin{equation} %
  \label{E:finiteness_of_the_expectation_hypothesis}
% \max\{\E S^{k+1}(t), 
\E S^{k-1}(t)S_m(t)S_{l}(t)<\infty, \quad t \in [0,1/\|\bx\|^2),\ \ m,l\geq 2,
\end{equation}
by induction in $k$. 
%zzzz Note that if $m=2$ then $l$ can only be $1$, so the hypothesis reads $\E S^{k+1}(t)<\infty$.
Step one serves as the basis, since  it is~\eqref{E:finiteness_of_the_expectation_hypothesis} for $k=1$. 
We left to the reader the even case ($n+m\geq 4$) from step one, and note that we already proved $\E S^2(t)<\infty$ in \cite{Konarovskyi:MC:2020} via a different argument.

We next assume that~\eqref{E:finiteness_of_the_expectation_hypothesis} is true for each $k \in [n]$ and check it for $k=n+1$. Let us apply $\Gamma$ to the product $s_2^ns_m$ (here and several times below we write $s_2$, $s_m$ for $s_2(\bx)$, $s_m(\bx)$, and use binomial formula in order to derive for $m\geq 2$:
\begin{align*}
  \Gamma s_2^n s_m(\bx)&= \sum_{ i=1 }^{ \infty } \sum_{ j= i+1 }^{ \infty } x_ix_j\left( \left(s_2+2x_ix_j\right)^n\left( s_m+\sum_{ l=1 }^{ m-1 } \binom{m}{l} x_i^lx_j^{m-l} \right)-s_2^ns_m \right)\\
  &= \sum_{ i=1 }^{ \infty } \sum_{ j= i+1 }^{ \infty } x_ix_j\left( \sum_{ k=0 }^{ n } \binom{n}{k} s_2^k(2x_ix_j)^{n-k}\left( s_m+\sum_{ l=1 }^{ m-1 } \binom{m}{l} x_i^lx_j^{m-l} \right) -s_2^ns_m\right)\\
  &= \sum_{ i=1 }^{ \infty } \sum_{ j= i+1 }^{ \infty } x_ix_j\Bigg( s_2^n\sum_{ l=1 }^{ m-1 } \binom{m}{l}x_i^lx_j^{m-l}+s_m\sum_{ k=0 }^{ n-1 } \binom{n}{k} 2^{n-k}s_2^kx_i^{n-k}x_j^{n-k}\\
&+\sum_{ k=0 }^{ n-1 }\sum_{ l=1 }^{ m-1 } 
\binom{m}{l} \binom{n}{k}
2^{n-k}s_2^kx_i^{n-k}x_j^{n-k}x_i^lx_j^{m-l}\Bigg).
\end{align*}
As before, we next exchange the order of summation to get that $\Gamma s_2^n s_m(\bx)$ equals
\begin{align*}
& s_2^n\sum_{ l=1 }^{ m-1 } \binom{m}{l} \sum_{ i=1 }^{ \infty } \sum_{ j= i+1 }^{ \infty } x_i^{l+1}x_j^{m-l+1} + s_m\sum_{ k=0 }^{ n-1 } \binom{n}{k}2^{n-k}s_2^k \sum_{ i=1 }^{ \infty } \sum_{ j= i+1 }^{ \infty }x_i^{n-k+1}x_j^{n-k+1}\\
&+\sum_{ k=0 }^{ n-1 }\sum_{ l=1 }^{ m-1 } \binom{m}{l}\binom{n}{k}2^{n-k}s_2^k \sum_{ i=1 }^{ \infty } \sum_{ j= i+1 }^{ \infty }x_i^{n-k+l+1}x_j^{n+m-l-k+1} .
\end{align*}
The middle term can be written already as
\begin{equation}
\label{E:middle}
s_m\sum_{ k=0 }^{ n-1 }\binom{n}{k} 2^{n-k-1}s_2^k\left(s_{n-k+1}^2-s_{2n-2k+2}\right).
\end{equation}
%Here $m$ can have any parity.
We denote the integer part
$\lfloor \frac{ m-1 }{ 2 }\rfloor$
by $\tilde{m}$, and if $\tilde{m}=0$ (meaning $m=2$) the sum from $1$ to $\tilde{m}$ is set to zero. 
With this in mind, again due to binomial symmetry, the first term above becomes
$$
s_2^n\sum_{ l=1 }^{ \tilde{m} } \binom{m}{l}\sum_{ i\not= j } x_i^{l+1}x_j^{m-l+1} + \I_{2\N}\left(m\right)\binom{m}{\sfrac{m}{2}}s_2^n\sum_{ i\not= j } x_i^{m/2+1}x_j^{m/2+1}/2,
$$
while the third term in the above sum (expression for $\Gamma s_2^n s_m(\bx)$) becomes
\begin{align*}
&\sum_{ k=0 }^{ n-1 }\sum_{ l=1 }^{ \tilde{m} }  \binom{m}{l} \binom{n}{k}2^{n-k}s_2^k \sum_{ i\not= j }x_i^{n-k+l+1}x_j^{n+m-l-k+1}\\
&+\I_{2\N}(m) \binom{m}{\sfrac{m}{2}}\sum_{ k=0 }^{ n-1 } \binom{n}{k}
2^{n-k-1}s_2^k \sum_{ i\not= j }x_i^{n-k+m/2+1}x_j^{n-k+m/2+1}\ .
\end{align*}
Now it suffices to observe that 
\begin{equation}
\label{E:first}
\sum_{ i\not= j } x_i^{l+1}x_j^{m-l+1} =s_{l+1}s_{m-l+1}-s_{m+2}, \ \
\sum_{ i\not= j } x_i^{m/2+1}x_j^{m/2+1}/2 = s_{m/2+1}^2-s_{m+2}, 
\end{equation}
and similarly that
\begin{align}
\sum_{ i\not= j }x_i^{n-k+l+1}x_j^{n+m-l-k+1}= s_{n-k+l+1}s_{n+m-l-k+1}-s_{2n-2k+m+2},\nonumber \\
\sum_{ i\not= j }x_i^{n-k+m/2+1}x_j^{n-k+m/2+1}= s_{n-k+m/2+1}^2-s_{2n-2k+m+2}, \label{E:third}
\end{align}
where $m/2$ above is assumed to be an integer in (\ref{E:first}--\ref{E:third}).
The reader will now easily see from \eqref{E:middle}--\eqref{E:third} and previous discussion that $ \Gamma s_2^n s_m(\cdot)$ can be written as a difference of two non-negative functions $f_2(\cdot)$ and $f_1(\cdot)$, where $f_2$ is a finite sum of positive multiples of 
 $s_2^ns_{l+1}s_{m-l+1}$, with $l \in [\lfloor m/2\rfloor]$,
 as well as positive multiples of $s_m s_2^k s_{n-k+1}^2$ with $k \in [n-1]$, and other similar terms. 
Furthermore it is important here that $f_1$ is a finite sum of positive multiples of terms of the form $s_2^{n-1}s_2 s_{m+2}$, or $s_2^k s_m s_{2n-2k+2}$ or $s_2^k s_{2n-2k+m+2}$ with $k\in \{0\}\cup [n-1]$.
 Therefore the induction hypothesis \eqref{E:finiteness_of_the_expectation_hypothesis}, together with monotonicity of each process $S_k(t)$ will guarantee the condition
 \[
  \E\int_{ 0 }^t  f_1(\bX(r))dr<\infty, 
\]
 of Lemma~\ref{lem_recurrent_property}
as in step one of the proof.
It seems simpler here and in the next paragraph to treat the case $m=2$ (where only the middle summand (\ref{E:middle}) exists) separately.

Hence 
\[
  \E\int_{ 0 }^t  f_2(\bX(r))dr<\infty, 
\]
for every $t \in (0,1/\|\bx\|^2)$. In particular,
\[
  \E \int_{ 0 }^{ t } S^{n}(r)S_{l+1}(r)S_{m-l+1}(r)dr<\infty, \quad \forall l \in [\lfloor m/2\rfloor]. 
\]
Since $m\geq 2$ was arbitrary and $(S_m(t), t \geq 0)$ is monotone non-decreasing for each $m\geq 2$, we arrive to~\eqref{E:finiteness_of_the_expectation_hypothesis} for $n+1$, and therefore for all $n\in \N$. Note that any two $(l_1,m_1)$, where $2\leq l_1\leq m_1$ can be represented as $(l+1,m-l+1)$ for some $m\geq 2$, $l \in [\lfloor m/2\rfloor]$. The statement of the proposition directly follows from~\eqref{E:finiteness_of_the_expectation_hypothesis}
with $m=2$ and $l=1$.
\end{proof}

\subsection{Extension of Theorem \ref{the_finiteness_of_fourth_moment} to all times}
In this section we present a ``finite modification argument'' which ends the proof.
We wish to warn the reader that, unlike most of the reasoning written in previous sections, this part of the proof is given in the appendix to \cite{Konarovskyi:MC:2020} for the special case $n=4$. 
Since in \cite{Konarovskyi:MC:2020} we used different, and more complicated notation, adapted to the study of stochastic block model and its continuum counterparts, it seems reasonable to also provide a sketch here using our current notation.

Let $\bX(\cdot;\bx)$ be the \MC\ started at time $0$ from $\bx \in \cvd$.
We know that with probability one, for all $t\geq 0$, $\bX(t;\bx) \in \cvd$, and in addition we know that if $\|\bx\|^2 t <1$, then for each $n\geq 2$ and $t\geq 0$
$$
\E(S_2(t;\bx))^{n/2} = \E \|\bX(t;\bx)\|^n <\infty.
$$
Now take any $\bx\in \cvd$ and $t\geq 1/\|\bx\|^2$ and
and let $m,M \in \N$ sufficiently large so that the vector
  \begin{equation}
  \label{def_x_grinded}
    \bx^g=\left( \frac{ x_1 }{ M },\dots,\frac{ x_1 }{ M },\frac{ x_2 }{ M },\dots,\frac{ x_2 }{ M },\dots,\frac{ x_m }{ M },\dots,\frac{ x_m }{ M },x_{m+1},x_{m+2}, \dots \right)\!,
  \end{equation}
  obtained by ``grinding'' the first $m$ components (blocks) of $\bx$ each into $M$ new components (blocks) of equal mass, has sufficiently small $l^2$ norm. More precisely, we take $m,M \in \N$ so that
  \[
    t\left\|\bx^g\right\|^2= t\left(\frac{ x_1^2 }{ M }+ \frac{ x_2^2 }{ M }+\dots+\frac{ x_m^2 }{ M }+ x_{m+1}^2+x_{m+2}^2+\dots\right)< \frac{1}{ 2 }.
  \]
  Then $\E \|\bX(2t;\ord(\bx^g))\|^n< +\infty$ due to Proposition~\ref{pro_finiteness_of_fourth_moment}.

For blocks with indices $i_1,i_2,\ldots, i_k$ we say that they {\em connect directly} at time $t$ if $G_t(\bx;\bA)\cap \{i_1,\ldots,i_k\}$ is a connected graph.
Let us assume that $x_1>x_2>x_3\ldots >x_m$ and that $\{x_i/M\}_{i\leq m} \cap \{x_k\}_{k\geq m+1}=\emptyset$, the argument is entirely analogous (but more tedious to write) otherwise. 
Note that the event $A$ on which at time $t$ the $M$ initial blocks of mass $x_1/M$ connect directly, the $M$ initial blocks of mass $x_2/M$ connect directly, $\ldots$, and the $M$ initial blocks of mass $x_m/M$ connect directly, has strictly positive probability. Of course there will be (infinitely many) other merging events occurring during $[0,t]$, which will involve these and other initial blocks. But these extra mergers only help in increasing the $l^2$ norm of $\bX$ at time $t$, and subsequently at time $2t$.
%Of course, during $[t,2t]$ additional merging will occur. 
It is not hard to see that the Markov property of the \MC\ implies
\begin{align*}
\infty>\E \|\bX(2t;\ord(\bx^g))\|^n &= \E \left(\E \left( \|\bX(2t;\ord(\bx^g))\|^n | \FF_t\right)\right) \\
&\geq 
\E \left(\E \left( \|\bX(2t;\ord(\bx^g))\|^n | \FF_t\right) \I_A\right)\\
&= \E \left(\E \|\tilde{\bX}(t;\bX(t;\ord(\bx^g))) \|^n \I_A\right),
 \end{align*} 
where $\tilde{\bX}(\cdot;\bX(t;\ord(\bx^g)))$ evolves, conditionally on $\FF_t$, as the \MC\ started from $\bX(t;\ord(\bx^g))$.
From previous discussion we see that on $A$ the random variable $\|\tilde{\bX}(t;\bX(t;\ord(\bx^g)))\|$ stochastically dominates $\|\bX(t;\bx)\|$ from above, therefore it is impossible that $\E\|\bX(t;\bx)\|^n=\infty$.
By varying $t\geq 1/\|\bx\|^2$ and recalling Proposition \ref{pro_finiteness_of_fourth_moment} we obtain Theorem \ref{the_finiteness_of_fourth_moment}.
\begin{remark}
As already mentioned, the above argument was written in detail in~\cite{Konarovskyi:MC:2020} using a graphical construction and notation analogous to that from Section~\ref{sub:graph_construction}. 
In the construction of $(G_t(\bx;\bA),\, t\geq 0)$ the family of edges arriving during $[0,t]$ and the family of edges arriving during $(t,2t]$ are mutually independent, implying the Markov property of $\bX$. 
Event $A \in \FF_t$ is independent from the $\sigma$-field generated by the edges connecting before time $t$
 pairs of blocks with masses $x_i/M$ and $x_j/M$ where $i\neq j$, the edges connecting pairs of blocks such that at least one of the blocks is not among the $M \cdot m$ ``crumbs'' with masses listed as the first  $M \cdot m$ components of $\bx^g$, as well as the edges arriving after time $t$. 
The facts that edges are only accumulating (and never deleted) over time, and that the $l^2$ norm is monotone increasing with respect to the subgraph relation, gives the key stochastic domination property used above.
\end{remark}

\section{Consequences for excursion processes}
\label{sec:conseq_excursions}
Recall the notation from Section \ref{sec:preliminaries} leading to the statement of Theorem~\ref{the_finiteness_of_the_second_moments_of_emc}.
In this section we give the proof of this theorem, starting with the case of negative (coalescent time) $t$. 

\begin{lemma} %negative times
  \label{lem_negative_times}
  Let the constant $D_2$ be defined in Lemma~\ref{lem_finitenes_of_expected_S_n}. Then for every $t<0$ the inequality 
  \[
    \E \sum_{ \gamma \in \Gamma^{\kappa,t,\bc} } |\gamma|^2\leq  - \frac{ D_2 }{ t }
  \]
 is satisfied for every $(\kappa,t,\bc) \in \cI$.
\end{lemma}

\begin{proof} %
  Let $\bX(t;\bx)$ be the MC\ started at time 0 from $\bx \in \cvd$.  According to Lemma~8, Proposition~7 and Theorem~3 of~\cite{Aldous:1998}, there exists a sequence $\bx^{n} \in \cvd$ such that $\|\bx^n\| \to 0$ and 
  \[
    \bX\left( \frac{1}{ \|\bx^n\|^2 }+t ; \bx^n\right) \to \bZ(t) \quad \mbox{in}\ \ \cvd,
  \]
  weakly as $n\to\infty$, where $\bZ(t)$ is distributed as the ordered sequence of lengths of excursions from $\Gamma^{\kappa,t,\bc}$. To be more precise, the sequence $\bx^{n}$, $n\geq 1$, can be chosen as follows. If $\kappa>0$ then $\bx^n$ consists of $n$ entries of size $\kappa^{-1/3}n^{-2/3}$, preceded by entries $(c_1 \kappa^{-2/3}n^{-1/3},\dots,c_{l(n)}\kappa^{-2/3}n^{-1/3})$, where $l(n) \to \infty$ sufficiently slowly. In the case $\kappa=0$ and $\bc \in \Lspec$, one can take $\bx^n$ to consist of entries $(c_1n^{-1/3},\dots,c_{l(n)}n^{-1/3})$, where $l(n) \to \infty$ fast enough so that $\sum_{ i=1 }^{ l(n)}c_i^2  \sim n^{1/3}$ (see the proof of Lemma~8~\cite{Aldous:1998}).

  We note that 
  \[
    \left( \frac{1}{ \|\bx^n\|^2 }+t \right)\cdot \|\bx^n\|^2=1+t \|\bx^n\|^2<1
  \]
  for $t<0$. Therefore, we may use Lemma~\ref{lem_finitenes_of_expected_S_n} to estimate the expectation 
  \[
    \E \left\|\bX\left( \frac{1}{ \left\|\bx^n\right\|^2 }+t;\bx^n \right)\right\|^2< \frac{ D_2 \|\bx^n\|^2 }{ -t\|\bx^n\|^2 }= - \frac{ D_2 }{ t }.
  \]
  Passing to the limit as $n\to\infty$ and using Fatou's lemma and Skorohod Theorem~3.1.8~\cite{Ethier:1986}, we obtain the statement of the lemma.
\end{proof}

\begin{remark}
It is somewhat surprising that the proof for non-negative times turns out to be less direct. 
A technical obstacle is that Lemma~\ref{lem_finitenes_of_expected_S_n}  cannot  apply any longer, since the upper bound used above diverges at $t=0$.
An obstacle in practice is that for positive $t$ the auxiliary process $W^{\kappa,t,\bc}$ has for small positive $s$ an ``extra push'' in terms of a positive inhomogeneous drift (if $\kappa >0$ this push has value $t-\kappa s$ at time $s$) which can (and does) increase the length of the initial (size-biased ordered) excursions of $B^{\kappa,t,\bc}$.
This increase does not change the finiteness of the square of $l^2$ norm almost surely. 
And the same should be true for the mean. 
\end{remark}

Let us assume that Theorem \ref{the_finiteness_of_the_second_moments_of_emc} fails, or equivalently, that
for some $\kappa,\bc$ and $t\geq 0$ it is true that
$$
   \E\sum_{ \gamma \in \Gamma^{\kappa,t,\bc} }   |\gamma|^2=\infty.
$$
Without loss of generality we may assume that $t>0$ (this is the \MC\ time, and its $l^2$ norm increases in time).
For the same reason, for this $(\kappa,t,\bc)$ and any $a>0$
\begin{equation}
\label{E:thm4point1contra}
   \E\sum_{ \gamma \in \Gamma^{\kappa,t + a,\bc} }   |\gamma|^2=\infty.
\end{equation}
Recall the ``collor and collapse'' (denoted by ${\rm  COL}$) operation from \cite{Aldous:1998} Section~5. 
In words, ${\rm COL}(\bx;\bc)$ is obtained by Poisson marking of the blocks (with respective masses $x_1, x_2,\ldots$) by points of one or more (or countably many) colors (to each $j$ such that $c_j>0$ correspond marks of the $j$th color, they are distributed at rate $c_j$ per unit mass, and the point processes of marks are independent over $j$) and then simultaneously merging together any pair of original blocks which have at least one mark of same color.  
If $\bX$ is random (this will be true below), ${\rm COL}(\bX;\bc)$ always supposes that the point processes of marks are not only mutually independent, but also independent from $\bX$. 
Similarly, recall the ``join'' (denoted by $\bowtie$) operator:
for  $\bx$ and $\by$ in $\cvd$, $\bx \bowtie \by \in \cvd$ is defined as the non-increasingly ordered listing of all the components from $\bx$ and from $\by$.

% and consider $\bc \bowtie c^*$.
Let $(X^{\kappa, \bc}(u), u \in \R)$ be eternal \MC\ such that, for each $u$, $X^{\kappa, \bc}(u)$ has the law equal to that of the vector of ordered excursion lengths of $B^{\kappa,u,\bc}$.
Then ${\rm COL}(X^{\kappa, \bc};c^*)$ is the eternal \MC\ such that ${\rm COL}(X^{\kappa, \bc};c^*)(u)=
{\rm COL}(X^{\kappa, \bc}(u);c^*)$ equals in law to the ordered excursion length vector of $B^{\kappa,u+(c^*)^2,\bc\bowtie c^*}$.
Furthermore, ${\rm COL}(X^{\kappa, \bc};(c_1^*,c_2^*,\ldots,c_k^*))$ is the eternal \MC\ such that its law at time $u$ is that of  the ordered excursion lengths of $B^{\kappa,u+\sum_{i=1}^k(c_i^*)^2,\bc\bowtie (c_1^*,\ldots,c_k^*)}$.

Now for $\kappa,t$ and $\bc$ fixed above, supposing that $\bc \in \cvt\setminus \cvd$, consider $u:=-t <0$ and find the smallest $m$ such that $\sum_{i=1}^m (c_i^*)^2\geq 2t$.
On the one hand, we know from Lemma~\ref{lem_negative_times} that 
$$
 \E\sum_{ \gamma \in \Gamma^{\kappa,u,(c_{m+1},c_{m+2},\ldots)} }   |\gamma|^2 < \frac{D_2}{-u}= \frac{D_2}{t}.
$$
On the other hand, we know from the just made observations that 
$$
 \E\sum_{ \gamma \in \Gamma^{\kappa,t + a,\bc} }   |\gamma|^2 = \infty,
$$
where $a=\sum_{i=1}^m (c_i^*)^2-2t\geq 0$,
and that the vector of ordered  lengths  $|\gamma|$, where $\gamma$ ranges over $\Gamma^{\kappa,t + a,\bc} $ has the same law as ${\rm COL}((X^{\kappa,u,(c_{m+1},c_{m+2},\ldots)});(c_1,\ldots,c_m))$.
Recall that $\Gamma^{\kappa,u,(c_{m+1},c_{m+2}\ldots)}$ is the family of excursions of $B^{\kappa ,u,(c_{m+1},c_{m+2},\ldots)}$.
To summarize, we found parameters $\kappa$, $\bar\bc:=(c_{m+1},c_{m+2},\ldots)$ and a negative time $u=-t$ such that the $l^2$ norm of
$\bX^{\kappa,\bar\bc}(u)$
%before applying ${\rm COL}(\cdot;(c_1,\ldots,c_m))$ 
 has finite expectation, 
while after applying ${\rm COL}(\cdot;(c_1,\ldots,c_m))$  the expectation 
of the same quantity becomes infinite.

We can assume (by combining all the finitely many colors into one) WLOG that $m=1$ in the just constructed example. Let us denote again $c_1=c_m$ by $c^*$.
It is interesting here that
coloring $\bX^{\kappa,\bar\bc}(u)$ with intensity $c^*$ or higher yields infinite mean of the $l^2$ norm, while coloring $\bX^{\kappa,\bar\bc}(u)$ with intensity equal to a positive fraction of $c^*$ (this could be $c^*/2$ or $c^*/10^6$, the conclusion will be the same) yields a finite mean of the $l^2$ norm, as long as $u+b<0$ where $b/(c^*)^2$ is the square of this sufficiently small fraction.
This fact is not only counter-intuitive, but also impossible as the following comparison argument shows.

We use a calculus fact: for each $\alpha>0$ the map $\varphi_\alpha: x \mapsto \frac{1-\exp{\{-\alpha c^* x\}}}{1-\exp{\{-c^* x\}}}$ from $(0,\infty)$ to $[0,1]$ admits a continuous extension at $0$ with value $\varphi_\alpha(0):= \alpha$, and satisfies $\lim_{x\to \infty} \varphi_\alpha(x)=1$, therefore 
\begin{equation}
\label{E:varalphastar}
1=
\sup_{x\in [0,\infty)} \varphi_\alpha(x) > \inf_{x\in [0,\infty)} \varphi_\alpha(x) = \min_{x\in[0,\infty)}\varphi_\alpha(x)  =:\varphi_\alpha^* >0.
\end{equation}
The ${\rm COL}$ operation changes the mass of only one (colored) block, it simultaneously deletes all the blocks which merge due to coloring. 
So if the mean $l^2$ norm after coloring is infinite (resp.~finite), it must be due to the fact that the mass of the colored block squared has infinite (resp.~finite) expectation.
In the case of intensity $c^*$ this quantity has value
$$
\sum_i \E (X_i(u))^2 (1-e^{-c^* X_i(u)}) + 
\sum_{i<j} \E X_i(u) X_j(u)(1-e^{-c^*  X_i(u)})(1-e^{-c^*  X_j(u)})=\infty. 
$$
Similarly, if the coloring intensity is $\alpha c^*$ for $\alpha$ suffiiently small, this quantity is
$$
\sum_i \E (X_i(u))^2 (1-e^{-\alpha c^* X_i(u)}) + 
\sum_{i<j} \E X_i(u) X_j(u)(1-e^{-\alpha c^* X_i(u)})(1-e^{-\alpha c^* X_j(u)})<\infty
$$
Due to \eqref{E:varalphastar} the two quantities above must be of the same order, which leads to a contradiction.
This shows  the statement of Theorem \ref{the_finiteness_of_the_second_moments_of_emc}  for $\bc \in \cvt\setminus \cvd$ and $\kappa \geq 0$.

If $\bc \in \cvd$ we cannot use for large times $t$ the same trick of ``stepping  sufficiently far back in time''  and then coloring.
However, we know that here it must be $\kappa >0$, and a clear advanatge here is that the sum of Brownian motion and a concave parabola is a convenient process for precise estimation. 
The argument given below includes stronger estimates than necessary for ending the proof of Theorem \ref{the_finiteness_of_the_second_moments_of_emc}.
They come at little additional cost, and might be  useful for further studies. 
The only restriction on $\bc$ is that it is a vector in $\cvt$. 
In particular, the argument below reproves the theorem in the case where $\kappa>0$ and $\bc \in \cvt \setminus \cvd$. 

\medskip
\begin{proof}[Proof of Theorem~\ref{the_finiteness_of_the_second_moments_of_emc} assuming $\kappa>0$.]
  Let $(\kappa,t,\bc) \in \cI$ be fixed.  Since $S_2(t)$ is monotone (non-decreasing, in fact increasing) in $t$, we may assume that $t$ is strictly positive.

Here we present a different way of exploiting the (uniform) bound from Lemma \ref{lem_negative_times}.
We introduce 
$t':=\frac{ 2t }{ \kappa }$ and
\begin{equation} %definition of sigma kappa
  \label{equ_definition_of_sigma_kappa}
  \begin{split}
    \sigma_{\kappa,t,\bc}&= \inf\left\{ s\geq 0:\ B^{\kappa,t,\bc}\left(t' +s\right)\leq 0 \right\}\\
    &= \inf\left\{ s\geq 0:\ W^{\kappa,t,\bc}\left(t'+s \right)= \min\limits_{ r \in [0,t'] }W^{\kappa,t,\bc}(r) \right\}.
  \end{split}
\end{equation}

\begin{proposition} %finiteness of stopping time
  \label{pro_finiteness_of_stopping_time}
  For every $(\kappa,t,\bc) \in \cI$, $t,\kappa>0$, and $n \in \N$, $\E\sigma_{\kappa,t,\bc}^n< +\infty$.
\end{proposition}
The proof is postponed until Section~\ref{sec:proof_pro_finiteness_of_stopping_time}.
The auxilliary time $t'$ is convenient for our purposes since 
% for the auxilliary process $W^{\kappa,t,\bc}$ 
 the parabola $s \mapsto ts - \sfrac{\kappa}{2}s^2$ starts to decrease at $t'/2$ and turns negative right after $t'$.
We split the family of excursions $\Gamma^{\kappa,t,\bc}$ into three subfamilies, according to whether they end before $t'$, start after $t'$ or traverse $t'$. 
The contributions coming from the first subfamily are easily controlled, those from the second family will be handled due to a comparison with a $(\kappa,\bar{t},\bar{\bc})$--setting where $\bar{t}$ is negative.
The third family clearly consists of a single (random) element of $\Gamma^{\kappa,t,\bc}$ which traverses (or includes) $t'$, and Proposition \ref{pro_finiteness_of_stopping_time} is used to bound the second moment of its length.

 %
%We remind that 
%  \[
%    \E\sum_{ \gamma \in \Gamma^{\kappa,t,\bc} }   |\gamma|^2< \infty,
%  \]
%  for every $t<0$, according to Lemma~\ref{lem_negative_times}. 
  
As just explained, we denote by $\Gamma_{>t'}$ (resp. $\Gamma_{<t'}$) the subfamily of excursions $\gamma$ of the process $B^{\kappa,t,\bc}$ satiffying $l(\gamma)>t'$ (resp. $r(\gamma)<t'$). Let also $\sigma=\sigma_{\kappa,t,\bc}$ be defined by~\eqref{equ_definition_of_sigma_kappa}.
  Denoting the excursion of $B^{\kappa,t,\bc}$ which traverses $t'$ by $\gamma_0$ we can trivially estimate 
  \begin{align*}
    \sum_{ \gamma \in \Gamma^{\kappa,t,\bc} }   |\gamma|^2&= \sum_{ \gamma \in \Gamma_{<t'} }   |\gamma|^2+|\gamma_0|^2+\sum_{ \gamma \in \Gamma_{>t'} }   |\gamma|^2\leq (t')^2+(t'+\sigma)^2+\sum_{ \gamma \in \Gamma_{>t'} }   |\gamma|^2\\
   &\leq 3(t')^2+2\sigma^2+\sum_{ \gamma \in \Gamma_{>t'} }   |\gamma|^2.
  \end{align*}
  Therefore the finiteness of $\E \sum_{ \gamma \in \Gamma^{\kappa,t,\bc} } |\gamma|^2$ would immediately follow from Proposition~\ref{pro_finiteness_of_stopping_time} with $n=2$ and the finiteness of $\E \sum_{ \gamma \in \Gamma_{>t'} }   |\gamma|^2$.

In order to show that $\E \sum_{ \gamma \in \Gamma_{>t'} }   |\gamma|^2<\infty$,
  we introduce the following time-shifted processes:
  \begin{align*}
    W_{t'+}(s):&= \tilde{W}_{t'+}(s)+V_{t'+}(s), \quad s\geq 0,\\
    B_{t'+}(s):&= W_{t'+}(s)-\min\limits_{ r \in [0,s] }W_{t'+}(r), \quad s\geq 0,
  \end{align*}
  where 
  \begin{align*}
    \tilde{W}_{t'+}(s):=\tilde{W}^{\kappa,t}(t'+s)-\tilde{W}^{\kappa,t}(t')=\sqrt{ \kappa }\left( W(t'+s)-W(t') \right)-ts- \frac{1}{ 2 }\kappa s^2,
  \end{align*}
  and 
  \[
    V_{t'+}(s):=V^{\bc}(t'+s)-V^{\bc}(t')=\sum_{ i=1 }^{ \infty } \left( c_i \I_{\left\{ t'<\xi_i\leq t'+s \right\}}-c_i^2s \right).
  \]
In words, the path of $W_{t'+}$ is obtained from the path of $W^{\kappa,t,\bc}$ by translating the origin to the point $(t', W^{\kappa,t,\bc}(t'))$ and ignoring the negative times in this new coordinate system. Then $B_{t'+}$ is obtained from $W_{t'+}$ by the usual reflection above past-minima.

A simple computation shows that $B^{\kappa,t,\bc}(t'+s)=B_{t'+}(s)$ for every $s\geq \sigma$. This can also be verified from a figure depicted the just described coupling of paths of $W^{\kappa,t,\bc}$ and $W_{t'+}$, and therefore of $B^{\kappa,t,\bc}$ and $B_{t'+}$.

Hence, 
  \begin{equation} %estimate of sum of length of excursions t and t plus
  \label{equ_estimate_of_sum_of_length_of_excursions_t_and_t_plus}
    \sum_{ \gamma \in \Gamma_{>t'} }   |\gamma|^2\leq \sum_{ \gamma \in \Gamma_{t'+} }   |\gamma|^2,
  \end{equation}
  where $\Gamma_{t'+}$ denotes the set of excursions above $0$ of the non-negative process $B_{t'+}$.

  Let $(\eta_i)_{i\geq 1}$ be a family of independent Bernoulli distributed random variables, where  $\eta_i$ has success probability $e^{-c_it'}$, for each $i \geq 1$.
Assume that  $(\eta_i)_{i\geq 1}$ is independent of $W$ and $(\xi_i)_{i\geq 1}$. Since the distributions of $\I_{\left\{ t'<\xi_i\leq t'+s \right\}}$ and $\eta_i\I_{\left\{ \xi_i\leq s \right\}}$ coincide for each $i\geq 1$, we conclude that the process $V_{t'+}$ is equal in law to 
  \[
    \sum_{ i=1 }^{ \infty } \left( \eta_ic_i \I_{\left\{ \xi_i\leq s \right\}}-c_i^2s \right),\quad  s\geq 0.
  \]
 Hence the processes
  \[
    W_{\eta}(s):=\sqrt{ \kappa }W(s)-ts- \frac{1}{ 2 }\kappa s^2+\sum_{ i=1 }^{ \infty } \left(\eta_ic_i \I_{\left\{ \xi_i\leq s\right\}}-c_i^2s\right) ,\quad  s\geq 0,
  \]
  and $W_{t'+}$ also have the same law. 

  We observe that 
  \[
    \E\sum_{ i=1 }^{ \infty } c_i^2(1-\eta_i)=\sum_{ i=1 }^{ \infty } c_i(1-e^{-c_it'})\leq \sum_{ i=1 }^{ \infty } c_i^3 t'<\infty.
  \]
  Hence $\sum_{ i=1 }^{ \infty } c_i^2(1-\eta_i)<\infty$ almost surely.
\begin{remark}
It may seem that we do not really have to worry about $\sum_{ i=1 }^{ \infty } c_i^2(1-\eta_i)$ being finite, 
since we could simply drop $-(1-\eta_i)c_i^2 s$ from the $i$th term in the series above, thus making the process decrease less steeply (and have longer excursions), but we would still need to check that the remaining $\sum_{i=1}^\infty \left( \eta_ic_i \I_{\left\{ \xi_i\leq s \right\}}-\eta_i c_i^2s\right)$ converges conditionally, which is equivalent to checking the finiteness of $\sum_{ i=1 }^{ \infty } c_i^2(1-\eta_i)<\infty$. 
\end{remark}
We can now rewrite for  $s\geq 0$
  \begin{align*}
    W_{\eta}(s)&=  \sqrt{ \kappa }W(s)-ts- \frac{1}{ 2 }\kappa s^2 +\sum_{ i=1 }^{ \infty } \left(\eta_ic_i \I_{\left\{ \xi_i\leq s \right\}}-\eta_ic_i^2 s\right)-\sum_{ i=1 }^{ \infty } (1-\eta_i)c_i^2s\\
  &=  \sqrt{ \kappa }W(s)+\left(-t-\sum_{ i=1 }^{ \infty } (1-\eta_i)c_i^2\right)s- \frac{1}{ 2 }\kappa s^2 +\sum_{ i=1 }^{ \infty } \left( \eta_ic_i \I_{\left\{ \xi_i\leq s \right\}}-(\eta_ic_i)^2s\right).
  \end{align*}
Therefore the conditional law of $W_\eta$ given $\sigma(\eta_i,\ i\geq 1)$ 
 is the law of $W^{\kappa,-\bar{t},\bar{\bc}}$, where $\bar{t}=t+\sum_{ i=1 }^{ \infty } (1-\eta_i)c_i^2\geq t$ almost surely, and where $\bar{\bc}=\ord(\eta_1c_1,\eta_2c_2,\dots)$. Hence, recalling  \eqref{equ_estimate_of_sum_of_length_of_excursions_t_and_t_plus} and the just introduced coupling, and applying $\omega$-by-$\omega$ the uniform bound of Lemma~\ref{lem_negative_times} under the conditional expectation yields
  \begin{align*}
    \E \sum_{ \gamma \in \Gamma_{>t'} }   |\gamma|^2 &\leq \E \sum_{ \gamma \in \Gamma_{t'+} }   |\gamma|^2= \E \sum_{ \gamma \in \Gamma_{\eta} }   |\gamma|^2\leq \E\left[ \E\left( \sum_{ \gamma \in \Gamma_{\eta} } |\gamma|^2  \bigg|\sigma(\eta_i,\ i\geq 1) \right) \right]\\
    &\leq \E\left[\E \left(\sum_{ \gamma \in \Gamma^{\kappa,-\bar{t},\bar{\bc}} }   |\gamma|^2\bigg|\sigma(\eta_i,\ i\geq 1)\right)\right]\leq  \E\left[ \frac{ D_2 }{ \bar{t} } \right]\leq \frac{ D_2 }{ t }.
  \end{align*}
As already explained, this completes the proof of the theorem.
\end{proof}

\subsection{Proof of Proposition \ref{pro_finiteness_of_stopping_time}}
\label{sec:proof_pro_finiteness_of_stopping_time}
We need some auxiliary statements. The first one is very easy, yet we emphasize it here in order to facilitate its application in two longer computations which follow.
\begin{lemma}
  \label{lem_estimate_basic}
For every $\lambda,s>0$ and $\xi$ an exponential random variable with rate $c$, we have
$$
\E e^{\lambda \I_{\left\{ \xi_i\leq s \right\}} }  \leq 1+ cs \lambda e^\lambda.
$$
\end{lemma}
\begin{proof} %
Since $e^{\lambda \I_{\left\{ \xi_i\leq s \right\}} }  = \I_{\{\xi>s\}} + e^\lambda \I_{\{\xi\leq s\}}$, we have
$$
\E e^{\lambda \I_{\left\{ \xi_i\leq s \right\}} }  = e^{-cs} + e^\lambda (1-e^{-cs}) = 1 + (e^\lambda-1)(1-e^{-cs})
\leq 1+  \lambda cs e^\lambda.
$$
\end{proof}

\begin{lemma} %estimate of ev
  \label{lem_estimate_of_ev}
  For every $\bc \in \cvt$, $a,s>0$ and $m \in \N$, we have
  \[
    \ln \E e^{a V^{\bc}(s)}\leq a \sum_{ i=1 }^{ m } c_i+\sum_{ i=1 }^{ m } \ln (1+a sc_i^2)+a^2se^{ac_{m+1}}\|\bc\|_3^3.
  \]
\end{lemma}

\begin{proof} %
  Using Lemma \ref{lem_estimate_basic} with $\lambda =ac_i$ and $c=c_i$ for each $i$, and elementary calculus, we derive for each $m\in \N$ and each finite $k>m$
  \begin{align*}
%    \ln \E e^{a V^{\bc}(s)}&= \lim_{k\to \infty}
\sum_{ i=1 }^{ k } \ln \E e^{a\left( c_i\I_{\left\{ \xi_i\leq s \right\}}-c_i^2s \right)}&\leq \sum_{ i=1 }^{ k } \left[\ln \left( 1+asc_i^2e^{ac_i} \right) -asc_i^2\right]\\
  &\leq \sum_{ i=1 }^{ m } \ln \left( 1+asc_i^2e^{ac_i} \right)+\sum_{ i=m+1 }^{k }\left[\ln \left( 1+asc_i^2e^{ac_i} \right)-asc_i^2\right]\\
    &\leq \sum_{ i=1 }^{ m } \ln \left( (1+asc_i^2)e^{ac_i} \right)+as\sum_{ i=m+1 }^{k } c_i^2(e^{ac_i}-1)\\
    &\leq a \sum_{ i=1 }^{ m } c_i+\sum_{ i=1 }^{ m } \ln (1+a sc_i^2)+a^2se^{ac_{m+1}}\sum_{ i=m+1 }^{ k } c_i^3,
  \end{align*}
where in the last line we used $c_{m+1}\geq c_{m+2}\geq ...$ as a consequence of $\bc \in \cvt$.
The independence of $(\xi_i)_{i\geq 1}$ implies that
$$
\E e^{a \sum_{ i=1 }^{ k }   {\left( c_i\I_{\left\{ \xi_i\leq s \right\}}-c_i^2s \right)} }=
\E \prod _{ i=1 }^{ k }   e^{a\left( c_i\I_{\left\{ \xi_i\leq s \right\}}-c_i^2s \right)}=
\prod_{ i=1 }^{ k } \E e^{a\left( c_i\I_{\left\{ \xi_i\leq s \right\}}-c_i^2s \right)}.
$$
We know that $\sum_{ i=1 }^{ k }   {\left( c_i\I_{\left\{ \xi_i\leq s \right\}}-c_i^2s \right)} $ converges conditionally (but not absolutely unless $\bc\in \cvd$) to $V^\bc(s)$, almost surely, as $k\to \infty$.
Since $\bc \in \cvt$, the above estimate yields, for each fixed $m$, a uniform upper bound on 
$\E e^{a \sum_{ i=1 }^{ k }   {\left( c_i\I_{\left\{ \xi_i\leq s \right\}}-c_i^2s \right)} }$
over $k$, so that Fatou's lemma implies 
the stated claim.
\end{proof}

The next result shows that we can replace in Lemma \ref{lem_estimate_of_ev} the marginal of $V^\bc$ with a maximum of $V^\bc$ on a compact interval, however the upper bound in no longer nicely expressed in terms of the parameters ($a$, the compact interval, and $\bc$).
\begin{lemma} %finiteness of e min of V
  \label{lem_finiteness_of_e_min_of_v}
  For every $\bc \in \cvt$, and $t,a>0$ we have
  \[
    \E e^{a\max\limits_{ s \in [0,t] }|V^{\bc}(s)|}<\infty.
  \]
\end{lemma}

\begin{proof} %
  We recall that $V^{\bc}$ is a supermatringale with Doob-Meyer decomposition
  \[
    V^{\bc}(s)=M^{\bc}(s)-A^{\bc}(s), \quad s\geq 0,
  \]
  where 
  \[
    A^{\bc}(s)=\sum_{ i=1 }^{ \infty } c_i^2(s-\xi_i)^+, \quad s\geq 0,
  \]
  and where $M^{\bc}(s)=\sum_{ i=1 }^{ \infty } \left( c_i \I_{\left\{ \xi_i\leq s \right\}}-c_i^2(\xi_i\wedge s) \right)$, $s\geq 0$, is a martingale with
  \[
    \langle M^{\bc}\rangle_s=\sum_{ i=1 }^{ \infty } c_i^3 (\xi_i\wedge s), \quad s\geq 0,
  \]
as its predictable quadratic variation (for more details see~\cite[Section~2.1]{Aldous:1998}).
Using H\"older's inequality we now get
  \begin{equation} %estimate of e a max
  \label{equ_estimate_of_e_a_max}
    \begin{split} %
      \E e^{a \max\limits_{ s \in [0,t] }|V^{\bc}(s)|}&\leq \E\left[ e^{a \max\limits_{ s \in [0,t] }|M^{\bc}(s)|} \cdot e^{a \max\limits_{ s \in [0,t] }A^{\bc}(s)}\right]\\
      &\leq \left( \E e^{2a \max\limits_{ s \in [0,t] }|M^{\bc}(s)|} \right)^{ \frac{1}{ 2 }} \cdot  \left( \E e^{2a \max\limits_{ s \in [0,t] }A^{\bc}(s)} \right)^{ \frac{1}{ 2 }}.
    \end{split}
  \end{equation}
  Due to the monotonicity of $A^{\bc}$, the independence of $(\xi_i)_{i\geq 1}$, and Lemma \ref{lem_estimate_basic}
we obtain 
  \begin{align}
   \ln \E e^{2a \max\limits_{ s \in [0,t] }A^{\bc}(s)}&= \ln \E e^{2a A^{\bc}(t)}=\ln \E e^{2a \sum_{ i=1 }^{ \infty } c_i^2(t-\xi_i)^+}\leq \ln \E e^{2a \sum_{ i=1 }^{ \infty } c_i^2t \I_{\left\{ \xi_i\leq t \right\}}} \nonumber\\
   &= \sum_{ i=1 }^{ \infty } \ln \E e^{2atc_i^2 \I_{\left\{ \xi_i\leq t \right\}}} \leq \sum_{ i=1 }^{ \infty } \ln \left( 1+2at^2c_i^3e^{2atc_i^2} \right)\nonumber\\
&\leq 2at^2\sum_{ i=1 }^{ \infty } e^{2atc_i^2}c_i^3\leq 2at^2e^{2atc_1^2}\|\bc\|_3^3.
\label{E:exp_bound_norm_A}
  \end{align}
Since $\bc \in \cvt$,  this implies the finiteness of $\E e^{2a \max\limits_{ s \in [0,t] }A^{\bc}(s)}$. 
  
\medskip
It remains to check the finiteness of the first factor on the right hand side in~\eqref{equ_estimate_of_e_a_max}.
In order to do so, we will use the following Bernstein-type inequality for martingales with bounded jumps~\cite[p.899]{Shorack:1986} (see also Theorem~3.3~\cite{Dzhaparidze:2001} for general
square-integrable martingales): set
  \[
  \psi(r):=\frac{ 2 }{ r^2 }\left[ r \left(\ln (r+1)-1\right)+\ln (r+1) \right] = \frac{2}{r^2}\int_0^r \ln(1+y)\,dy , \quad r>0,
  \]
then for any $M$ a local martingale with jumps absolutely bouned by $K$, any time $s\geq 0$ and any two fixed levels $\lambda>0$ and $0<\tau<\infty$ we have
\begin{equation}
\label{E:Bernsteintype}
\p\left[ \sup_{s\in [0,t]} |M(s)| \geq \lambda,\  \langle M\rangle_t\leq \tau \right] \leq 2 \exp\left(-\frac{\lambda^2}{2\tau}\psi\left( \frac{\lambda c}{\tau}\right) \right).
\end{equation}
For our purposes we note that $M^\bc$ has jumps bounded by $c_1$, and that $\langle M^\bc \rangle_t$ is bounded by $t \sum_i c_i^3=t\|\bc\|_3^3$ almost surely.
We can thus apply \eqref{E:Bernsteintype} with $\lambda = r$ and $\tau= t\|\bc\|_3^3$ to get
$$\p\left\{ \max\limits_{ s \in [0,t] }|M^{\bc}(s)|\geq r \right\} = \p\left\{ \max\limits_{ s \in [0,t]} 
|M^{\bc}(s)|\geq r,\ \langle M^{\bc}\rangle_t\leq t\|\bc\|_3^3 \right\} \leq 2e^{-\frac{ r^2 }{ 2 t\|\bc\|_3^3 }\psi\left( \frac{ rc_1 }{ t\|\bc\|_3^3 } \right)}.
$$ 
Since $\psi(r) \geq 2(\ln(r+1)-1)/r$ we conclude that
\begin{equation}
\label{E:M_survival_estimate}
\p\left\{ \max\limits_{ s \in [0,t] }|M^{\bc}(s)|\geq r \right\} \leq 2e^{ -\frac{ r }{ c_1 }\left(\ln\left( 1+  \frac{ rc_1 }{ t\|\bc\|_3^3 } \right)-1\right)}.
\end{equation}
In words, the survival function of $\max\limits_{ s \in [0,t] }|M^{\bc}(s)|$ has superexponentially decreasing tails and now it is easy to see that  \eqref{E:M_survival_estimate} leads to
 \[
      \E e^{2a \max\limits_{ s \in [0,t] }|M^{\bc}(s)|} \leq 1+4a \int_{ 0 }^{ \infty } e^{2ar}\cdot e^{ - \frac{r}{ c_1 }\left(\ln \left( \frac{ rc_1 }{ t\|\bc\|_3^3 }+1 \right)-1\right)}dr< \infty,
\]
which together with  \eqref{equ_estimate_of_e_a_max} and \eqref{E:exp_bound_norm_A} yields the stated claim.
\end{proof}

\begin{proof}[End of the proof of Proposition~\ref{pro_finiteness_of_stopping_time}] %
  We fix $(\kappa,t,\bc)\in \cI$, $\kappa>0$, and recall that $t'= \frac{ 2t }{ \kappa }$. Then for every $s>0$ the following inequality is clearly true 
  \[
    \p\left\{ \sigma_{\kappa,t,\bc}\geq s  \right\}\leq \p\left\{ W^{\kappa,t,\bc}\left( t'+s \right)\geq \min\limits_{ r \in [0,t'] }W^{\kappa,t,\bc}(r) \right\}.
  \]
Here we are comparing the value of $W^{\kappa,t,\bc}\left( t'+s \right)$ for some (think large) positive $s$ to the value of the minimum of $W^{\kappa,t,\bc}$ on a fixed interval $[0,t']$. Our intention is likely already clear to the reader: use Lemmas \ref{lem_estimate_of_ev} and \ref{lem_finiteness_of_e_min_of_v} to ``tame'' the contribution of $V^\bc$ and let the Brownian (Gaussian) component be the driving force in estimating the probability on the right-hand-side in the previous inequality.
Since
  \begin{align*}
    W^{\kappa,t,\bc}\left( t'+s \right)=\sqrt{ \kappa }W\left( t'+s \right)+ts- \frac{1}{ 2 }\kappa s^2+V^{\bc}\left( t'+s \right),
  \end{align*}
  the intequality $W^{\kappa,t,\bc}\left( t'+s \right)\geq  \min\limits_{ r \in [0,t'] }W^{\kappa,t,\bc}(r)$ is equivalent to 
  \[
    \sqrt{ \kappa }\left(W\left( t'+s \right)-W\left( t' \right)\right)\geq -ts + \frac{1}{ 2 }\kappa s^2-U(s)-G,
  \]
  where 
  \begin{align*}
    U(s)&= V^{\bc}\left( t'+s \right) \quad \mbox{and} \quad G=\sqrt{ \kappa } \cdot W(t')- \min\limits_{ r \in [0,t'] }W^{\kappa,t,\bc}(r).
  \end{align*}
  We remark that the process $W'=(W(t'+s)-W(t'), \,s\geq 0)$ is a copy of standard Brownian motion, independent of $\sigma\{U,G\}$. So using conditioning and the well-known estimate for the standard normal survival function
  \[
   1-\Phi(x) \leq \sqrt{\frac{ 2 }{ \pi }} \cdot \frac{1}{ x }  \cdot e^{-\frac{ x^2 }{ 2}}, \ x> 0,
  \]
we can extend the previous bound for any $s\geq 2t'$ to
  \begin{align}
    \p\left\{ \sigma_{\kappa,t,\bc}\geq s \right\}&\leq \p\left\{ \sqrt{ \kappa }W'(s)\geq -ts+ \frac{1}{ 2 }\kappa s^2 -U(s)-G \right\}\nonumber\\
    &=\E\left[\p\left\{ \sqrt{ \kappa }W'(s)\geq ts -U(s)-G\, \big|\, \sigma\{U(s),G\} \right\}\right]\\
 &=\E\left[ 1-\Phi ( (ts -U(s)-G)/\sqrt{s \kappa } )\right]\nonumber\\
    &\leq  \sqrt{\frac{ 2 s \kappa  }{ \pi (1+ts/2)} }\,\E\left[ e^{- \frac{ \left( ts-U(s)-G \right)^2 }{ 2s \kappa }}\I_{\left\{ ts-U(s)-G>1 + ts/2 \right\}} \right]+ \E{ \I_{\left\{ ts-U(s)-G\leq 1+ts/2 \right\}} }\nonumber\\
    &\leq 2\sqrt{\frac{ \kappa  }{ t\pi } }\,\E\left[ e^{- \frac{ \left( ts-U(s)-G \right)^2 }{ 2s \kappa }}\right]+ \p\left\{ ts/2-U(s)-G\leq 1 \right\}.
\label{E:bound_two_summands}
  \end{align}
  Our goal is to show that both terms in \eqref{E:bound_two_summands} decrease exponentially fast as $s \to +\infty$. To this purpose we first estimate 
  \begin{align}
    \ln \E e^{- \frac{ \left( ts-U(s)-G \right)^2 }{ 2s \kappa }}&\leq \ln \E e^{- \frac{ t^2s^2 -2ts(U(s)+G)+(U(s)+G)^2}{ 2s \kappa }}\leq \ln \E e^{- \frac{ t^2s^2 -2ts(U(s)+G)}{ 2s \kappa }}\nonumber \\
    & =-\frac{ t^2s }{ 2 \kappa }+\ln \E \left[ e^{\frac{ t }{ \kappa }U(s) } \cdot e^{\frac{ t G }{ \kappa }} \right] \stackrel{\text{H\"older}}{\leq} -\frac{ t^2s }{ 2 \kappa }+  \ln\left(\E e^{\frac{ 2t }{  \kappa }U(s)}\right)^{1/2}\cdot \ln \left(\E e^{\frac{ 2tG }{  \kappa }}\right)^{1/2}\nonumber\\
    &= -\frac{ t^2s }{ 2 \kappa }+ \frac{1}{ 2 } \ln\E e^{\frac{ 2t }{ \kappa }U(s)}+ \frac{1}{ 2 } \ln \E e^{\frac{ 2tG }{ \kappa }}. \label{E:estimate_exp_U_G}
  \end{align}
Recalling the definition of $U(s)$ and Lemma~\ref{lem_estimate_of_ev} we get that the second summand above is bounded by
  \begin{align*}
     \frac{t}{ \kappa } \sum_{ i=1 }^{ m } c_i &+ \frac{1}{ 2 }\sum_{ i=1 }^{ m } \ln \left( 1+\frac{ 2t(t'+s)c_i^2 }{ \kappa } \right)+ \frac{ 2t^2(t'+s) }{ \kappa^2 }e^{\frac{ 2tc_{m+1} }{ \kappa }} \sum_{ i=m+1 }^{ \infty } c_i^3,
  \end{align*}
which for $s$ large is on the order of $m \cdot \ln s + s \sum_{ i=m+1 }^{ \infty } c_i^3$. So choosing $m$ large enough, this term can be made for all large $s$ smaller than a multiple of $t^2 s/8\kappa$.

The finiteness of $\E e^{\frac{ 2tG }{ \kappa }}$ follows similarly from Lemma~\ref{lem_finiteness_of_e_min_of_v}, the well-known fact that $\max\limits_{ r \in [0,t'] }W(s)\stackrel{d}{=}|W(t')|$, and the repeated application of the H\"older inequality. We leave the details to the reader. 
 % \begin{align*}
 %   \ln \E e^{\frac{ 2tG }{ \kappa }}&\leq \ln \E e^{t'\left( \sqrt{ \kappa }W(t')-\min\limits_{ r \in [0,t'] }W^{\kappa,t,\bc}(r) \right)}\leq \frac{1}{ 2 }\ln \E e^{2t' \sqrt{ \kappa }W(t')}+ \frac{1}{ 2 }\ln \E e^{-2t' \min\limits_{ r \in [0,t'] }W^{\kappa,t,\bc}(r)}\\
%    &\leq \frac{1}{ 2 }\ln \E e^{2t' \sqrt{ \kappa }W(t')}+ \frac{1}{ 2 } \ln \E e^{-2t' \min\limits_{ r \in [0,t'] }\left( \sqrt{ \kappa }W(s)+tr- \frac{1}{ 2 }\kappa r^2+V^{\bc}(s) \right)}\\
%    &\leq \frac{1}{ 2 }\ln \E e^{2t' \sqrt{ \kappa }W(t')}+ \frac{1}{ 4 }\ln \E e^{4t'\sqrt{ \kappa } \max\limits_{ r \in [0,t'] }|W(s)|}+ \frac{1}{ 4 }\ln \E e^{4t' \max\limits_{ r \in [0,t'] }|V^{\bc}(s)|}.
 % \end{align*}
We can now conclude that for all $s$ large enough the sum of the three terms in \eqref{E:estimate_exp_U_G} is bounded by $-t^2s/(4\kappa)$.

Let us consider the second summand in \eqref{E:bound_two_summands}.  Using Markov's and H\"older's inequalities, we estimate as before
  \begin{align*}
    \ln \p\{ ts/2-U(s)&-G\leq 1 \}\leq \ln \p\left\{ e^{U(s)+G}\geq e^{ts/2-1} \right\}\leq \ln \left( e^{-ts/2+1}\E e^{U(s)+G} \right)\\
    &\leq -\frac{ts}{2}+1+ \frac{1}{ 2 }\ln \E e^{2U(s)}+ \frac{1}{ 2 }\ln \E e^{2G}\\
    &\leq -\frac{ts}{2}+1+ \sum_{ i=1 }^{ m } c_i+ \frac{1}{ 2 }\sum_{ i=1 }^{ m } \ln \left( 1+2(t'+s)c_i^2 \right)\\
    &\qquad+2(t'+s)e^{2c_{m+1}}\sum_{ i=m+1 }^{ \infty } c_i^3+ \frac{1}{ 2 }\ln \E e^{2 G},
  \end{align*}
which can be bounded (via the same reasoning we applied in bounding \eqref{E:estimate_exp_U_G}) by $-\frac{ts}{4}$ for all large $s$. 
  
To summarize, we now know that, for all sufficiently large $s$, the survival probability
 $\p\left\{ \sigma_{\kappa,t,\bc}\geq s \right\}$ is dominated from above by $2\exp\{-s \min\{t/4, t^2/(4\kappa)\}\}$.
This multiple exponent (a function of $t$ and $\kappa$) is not the best (largest) possible, but here we are not interested in finding the optimal parameter.
Since exponential tails of the distribution are clearly sufficient for finite moment of any order, the proposition is proved.
\end{proof}

%%%%%%%%%%%%%%%%%%%%%%%%%%%%%%%%%%%%%%%%%%%%%%
%% Support information, if any,             %%
%% should be provided in the                %%
%% Acknowledgements section.                %%
%%%%%%%%%%%%%%%%%%%%%%%%%%%%%%%%%%%%%%%%%%%%%%
\begin{acks}[Acknowledgments]
The research presented in this paper was mostly conducted while the first author was employed at Hamburg University. 
%The authors would like to thank the anonymous referees, an Associate
%Editor and the Editor for their constructive comments that improved the
%quality of this paper.
\end{acks}

%%%%%%%%%%%%%%%%%%%%%%%%%%%%%%%%%%%%%%%%%%%%%%
%% Funding information, if any,             %%
%% should be provided in the                %%
%% funding section.                         %%
%%%%%%%%%%%%%%%%%%%%%%%%%%%%%%%%%%%%%%%%%%%%%%
\begin{funding}
The first author was partly supported by a visiting professor position from the University of Strasbourg and partly supported by the Deutsche Forschungsgemeinschaft (DFG, German Research Foundation) – SFB 1283/2 2021 – 317210226.  
%The first author was supported by NSF Grant DMS-??-??????.
%
%The second author was supported in part by NIH Grant ???????????.
\end{funding}

\end{document}